\documentclass{amsart} 
\usepackage{latexsym}
\usepackage{graphicx}
\usepackage{amsfonts}
\usepackage{amssymb}
\usepackage{amsmath}
\usepackage[mathscr]{euscript}
\usepackage{amscd}
\usepackage{amsthm}
\usepackage[all]{xy}

\newcommand{\mb}{\mathbf}
\renewcommand{\dot}{\centerdot}

\newcommand{\cM}{\mathcal{M}}
\newcommand{\sP}{\mathscr{P}}

\newcommand{\A}{\mathcal{A}}

\newcommand{\cC}{\mathcal{C}}

\newcommand{\cF}{\mathcal{F}}

\newcommand{\bN}{\mathbf{N}}
\newcommand{\sC}{\mathscr{C}}

\newcommand{\ul}{\underline}

\newcommand{\gc}{\widehat{\cC}}
\newcommand{\hf}{\dfrac{1}{2}}
\newcommand{\fto}{\xrightarrow}

\begin{document}

\numberwithin{equation}{section}

\newtheorem*{thm*}{Theorem 1.8}
\newtheorem{thm}[equation]{Theorem}
\newtheorem{prop}[equation]{Proposition}
\newtheorem{lem}[equation]{Lemma}
\newtheorem{cor}[equation]{Corollary}

\newtheoremstyle{example}{\topsep}{\topsep}%
     {}
     {}
     {\bfseries}
     {.}
     {2pt}
     {\thmname{#1}\thmnumber{ #2}\thmnote{ #3}}

   \theoremstyle{example}
   \newtheorem{nota}[equation]{Notation}
   \newtheorem{Defi}[equation]{Definition}
   \newtheorem{rem}[equation]{Remark}

\title[Coherence]{Infinite loop spaces, and coherence for symmetric monoidal bicategories}

\author{Nick Gurski}
\address{
School of Mathematics and Statistics,
University of Sheffield,
Sheffield, UK, S3 7RH
}
\email{nick.gurski@sheffield.ac.uk}
\author{Ang\'elica M. Osorno}
\address{
Department of Mathematics,
University of Chicago,
Chicago, IL 60637 USA
}
\email{aosorno@math.uchicago.edu}



\begin{abstract}
This paper proves three different coherence theorems for symmetric monoidal bicategories.  First, we show that in a free symmetric monoidal bicategory every diagram of 2-cells commutes.  Second, we show that this implies that the free symmetric monoidal bicategory on one object is equivalent, as a symmetric monoidal bicategory, to the discrete symmetric monoidal bicategory given by the disjoint union of the symmetric groups.  Third, we show that every symmetric monoidal bicategory is equivalent to a strict one.

We give two topological applications of these coherence results.  First, we show that the classifying space of a symmetric monoidal bicategory can be equipped with an $E_{\infty}$ structure.  Second, we show that the fundamental 2-groupoid of an $E_{n}$ space, $n \geq 4$, has a symmetric monoidal structure.  These calculations also show that the fundamental 2-groupoid of an $E_{3}$ space has a sylleptic monoidal structure.
\end{abstract}

\maketitle

\tableofcontents

\section*{Introduction}

Monoidal categories come in three basic types:  plain monoidal, braided monoidal, and symmetric monoidal.  Each kind of monoidal category also has its own accompanying coherence theory.  The basic results in the plain and symmetric case were worked out in \cite{maclane-assoc}, and the braided case appears in \cite{js}.  It should also be pointed out that \cite{js} re-examines the plain monoidal case, and gives a more complete treatment that includes coherence for monoidal functors as well as monoidal categories.  In each case, the fundamental coherence result is that any diagram built from the monoidal constraint isomorphisms, and satisfying some additional property, automatically commutes.  The key task is then to determine what additional properties are required of a diagram.

In the case of monoidal categories, the only additional property required is that the diagram be the image of a diagram in a free monoidal category generated by a set of objects.  (Note that this is a sufficient condition:  there is nothing preventing other diagrams in specific monoidal categories from commuting.)  In particular, this result implies that every diagram in a free monoidal category generated by a set of objects commutes.  While this condition might appear stringent at first sight, in fact it is very much not so.  The diagrams that this condition rules out are those which are only diagrams ``accidentally.''  Here is an example.  Assume we have a monoidal category with objects $x,y$ such that $xy=I=yx$.  Then we could ask if the following diagram commutes.
\[
\xy
{\ar^{a} (0,0)*+{(xy)x}; (25,0)*+{x(yx)} };
{\ar^{=} (25,0)*+{x(yx)}; (50,0)*+{xI} };
{\ar^{r} (50,0)*+{xI}; (50,-15)*+{x} };
{\ar_{=} (0,0)*+{(xy)x}; (0,-15)*+{Ix} };
{\ar_{l} (0,-15)*+{Ix}; (50,-15)*+{x} };
\endxy
\]
The coherence theorem for monoidal categories does not imply that this diagram commutes since we are required to use the equations $xy=I=yx$ to even construct the diagram in the first place.   Thus the condition of being the image of a diagram in a free object rules out the diagram relying on equations between objects.  Since braided and symmetric monoidal categories are monoidal by neglect of structure, this condition will always lurk in the background of any coherence theorem.

The coherence theorem for braided monoidal categories adds a (sufficient, but once again not necessary) geometric condition.  The additional structure in a braided monoidal category is a braid which is an isomorphism $\gamma_{xy}:xy \rightarrow yx$.  As the name implies, this serves to braid the object $x$ over the object $y$, and so the axioms that this collection of isomorphisms must satisfy are reminiscent of the relations in the standard presentation of the braid groups.  In particular, a morphism which is constructed from the coherence isomorphisms in a braided monoidal category can be assigned a geometric braid:  the associativity and unit isomorphisms produce the identity braid, and the isomorphisms $\gamma_{xy}$ are interpreted as one expects.  The coherence theorem in this case states that two such morphisms $f,g:a \rightarrow b$ are equal if their underlying geometric braids are equivalent as such.  In particular, the free braided monoidal category generated by a single object is equivalent
to the braided monoidal category $\mb{Br}$ which is the disjoint union of the braid groups.

The coherence theorem for symmetric monoidal categories is similar to that for braided monoidal categories in that it gives a sufficient condition for a diagram to commute based upon some invariant of the morphisms in that diagram.  A symmetric monoidal category adds the requirement that $\gamma_{yx}\gamma_{xy} = 1_{xy}$, and this means that a morphism constructed from the coherence constraints now has an underlying permutation instead of an underlying braid.  Hence the theorem is that a pair of parallel morphisms $f,g:a \rightarrow b$, each of which is constructed only out of coherence isomorphisms, are equal if their underlying permutations are equal.

One could then pursue coherence theorems for each of the various flavors of monoidal bicategories.  For plain monoidal bicategories, this was accomplished as a corollary to coherence for tricategories.  In \cite{GPS}, the authors prove that every monoidal bicategory is equivalent to a Gray-monoid.  Such an object is a monoid in the category of 2-categories and 2-functors equipped with Gray tensor product, and therefore has strict composition as well as a strictly associative and unital multiplication.  Gray-monoids do not have a notion of the tensor product of a pair of 1-cells, and there is only a 2-cell isomorphism between the composites
\[
(f \otimes 1) \circ (1 \otimes g) \cong (1 \otimes g) \circ (f \otimes 1),
\]
which itself then satisfies coherence axioms.  The diagrammatic form of coherence was established by the first author in his PhD thesis, and will appear in publication in \cite{gurski-tri}.  This form of coherence states that in a free monoidal bicategory, every diagram of constraint 2-cells commutes.

The basic coherence theory for braided monoidal bicategories was also established by the first author.  In \cite{gurski-braid}, it is shown that every diagram of constraint 2-cells in a free braided monoidal bicategory commutes.  In particular, this shows that the braiding is essentially a one-dimensional phenomenon.  This paper also shows that every braided monoidal bicategory is biequivalent to a strict braided monoidal bicategory in the sense of Crans \cite{crans}, in particular showing that the earlier definition of Baez and Neuchl \cite{BN} is not actually weaker than that of Crans if we work up to braided monoidal biequivalence.

Unlike the case of monoidal categories in which there are three varieties (plain, braided, and symmetric), for bicategories there are four:  plain, braided, sylleptic, and symmetric.  A sylleptic structure adds an invertible 2-cell between braiding twice and the identity, satisfying some naturality-type conditions, and a symmetric structure imposes a further axiom concerning braiding three times.  Neither of these kinds of monoidal bicategories have been studied from the perspective of coherence theory, and this paper treats the symmetric case.  Our main theorem is the following.
\begin{thm*}
In the free symmetric monoidal bicategory on a single object, every diagram of 2-cells commutes.  Equivalently, between every pair of parallel 1-cells there is either a unique invertible 2-cell or no 2-cells at all.  Moreover, parallel 1-cells are isomorphic if and only if they have the same underlying permutation.
\end{thm*}

In practice, it is usually the plain monoidal and the symmetric monoidal coherence theorems which are most often invoked.  The plain monoidal case is foundational, while many naturally-occurring, large categories which have monoidal structures are in fact symmetric.  Additionally, symmetric monoidal structures on small categories are useful in stable homotopy theory for producing spectra through the use of various K-theory machines \cite{maygeom, segal}.  Indeed, the coherence theorem is an essential ingredient of the proofs that the classifying space of a symmetric monoidal category is an $E_{\infty}$ space.  As other fields continue to incorporate more and more constructions of an essentially 2-categorical nature, one should expect this trend to continue, with the coherence theorems for monoidal bicategories and for symmetric monoidal bicategories to be the most useful in applications.

Here is an outline of this paper.  The first section collects together some background information and the statements of the coherence theorems.  This consists primarily of definitions we will need later, but we also give some examples of symmetric monoidal bicategories.  Most of these examples are relatively simple bicategories, and some of them have already been used in applications such as 2-K-theory. Since there are so many ways of stating coherence, we devote the third part of this section to identifying three variations of coherence that we will prove later.  The first is a theorem stating that all diagrams of 2-cells commute in free symmetric monoidal bicategories.  This statement of coherence is often the most important, as it can be used to prove that particular diagrams of constraint 2-cells in non-free symmetric monoidal bicategories commute by lifting them to a free object, determining that they commute there, and then concluding that the original diagram must commute as well since it is the 
image of a commuting diagram.  The second statement of coherence is that the free symmetric monoidal bicategory on one object is equivalent to the symmetric groups, seen as a discrete symmetric monoidal bicategory.  This is a direct consequence of the first form of coherence.  Finally, we show that every symmetric monoidal bicategory is equivalent to a strict one.

The second section of this paper is concerned with the relationship between symmetric monoidal bicategories and $E_{n}$ spaces for $n \geq 3$ (and hence between symmetric monoidal bicategories and $n$-fold loop spaces).  We prove results in both directions:  we construct an $E_{\infty}$ structure on the classifying space of a symmetric monoidal bicategory, and we show that the fundamental 2-groupoid of an algebra for an $E_{n}$ operad has a sylleptic structure when $n=3$ and a symmetric structure when $n \geq 4$.  The first of these results is proved using $\Gamma$-space techniques, and it is quite an interesting question to do so using operadic techniques as well; this would likely require a significant amount of 2-dimensional algebra, as the natural structures all seem to come as pseudo-algebras rather than algebras over an operad (or even pseudo-operad).  The theorems which give additional structure to the fundamental 2-groupoids of $E_{n}$ algebras are easier the larger $n$ is.  The symmetric case follows
immediately, while the sylleptic case requires some uninteresting but substantial calculation.  This is purely a consequence of the structure of the higher homotopy groups of configuration spaces.

We expect further topological applications to follow in the future.  The proof that the homotopy category of stable 1-types is equivalent to the homotopy category of Picard categories by Niles Johnson and the second author \cite{JO} uses the coherence theorem for symmetric monoidal categories in a fundamental way.  While extending this to dimension two will certainly involve additional complications, the coherence theorem for symmetric monoidal bicategories is a necessary first step in that program.

The third section is focused on the proof of the coherence theorems. The first part of this section focuses on certain technical aspects of the proof of the coherence theorem, and relies on the use of techniques from the theory of positive braids.  We use these algebraic results in order to replace the free symmetric monoidal bicategory on one object with a biequivalent one whose 1-cells are only the positive braids.  The second and third parts implement a rewriting strategy, and the fourth finishes the proof of our coherence theorems.

The strategy that we use is in some ways quite different from that used to prove coherence for symmetric monoidal categories.  There, Mac Lane \cite{maclane-assoc} had proved the theorem for the symmetric case before braided monoidal categories had even been defined as an independent structure of interest.  Our method, on the other hand, uses both the coherence results from the braided case as well as some purely algebraic results about positive braids.  Additionally, it should be clear that starting in $\S$3.1, our techniques rely heavily on the symmetric structure; in particular, it is not clear that the strategy adopted here could be used to prove a coherence theorem for sylleptic monoidal bicategories.  Since in the sylleptic case there are diagrams of 2-cells which do not commute (for example, the extra symmetry axiom), a coherence theorem for this structure would be quite interesting but also likely very difficult.

The authors would like to thank Joan Birman, Fred Cohen, Benson Farb, Niles Johnson, Mikhail Kapranov, and Peter May for conversations that contributed to the completion of this paper.  The first author would also like to thank the University of Chicago for its hospitality in November 2011, and the second author would like to extend a similar thank you to the University of Sheffield regarding a visit in September 2012.

\section{Background}

Here we collect together some background information.  The focus is on symmetric monoidal bicategories, so we give definitions of sylleptic and symmetric monoidal bicategories as well as the functors between them in the first part, while the second part gives some examples of symmetric monoidal bicategories that occur quite naturally.  The reader entirely unfamiliar with monoidal bicategories should look elsewhere first for the relevant definitions; our recommendations include \cite{GPS} and \cite{gurski-braid}.

\subsection{Definitions}

This section collects together the key definitions we will use in our proof of coherence, namely the definition of a symmetric monoidal bicategory, the definition of a symmetric monoidal functor, and the definition of the category of symmetric monoidal bicategories and strict symmetric monoidal functors between them.  In order to keep large diagrams to a minimum, we have assumed a rather high base of knowledge, namely the definition of a braided monoidal bicategory as well as the definition of braided monoidal functors between them.  The reader looking for more background, as well as these definitions, should consult \cite{gurski-braid}.

While the focus of this paper is on the symmetric case, we do prove a theorem about sylleptic monoidal bicategories in  $\S$\ref{pi}.  To be explicit, we remind the reader of that definition as well.

\noindent \textbf{Definition.}  A \textit{sylleptic monoidal bicategory} $X$ consists of an underlying braided monoidal bicategory $(X, \otimes, I, \mathbf{a}, \mathbf{l}, \mathbf{r}, \pi, \mu, \rho, \lambda, R, R_{-|-,-}, R_{-,-|-})$ together with an invertible modification $v: R \circ R \Rightarrow 1$ with components
$v_{xy}:R_{yx}R_{xy} \Rightarrow 1$ satisfying the following axioms; here we use the convention that if $\alpha$ is a 2-cell, then $\widehat{\alpha}$ denotes its mate. \\
\noindent \textbf{Syllepsis axioms:}
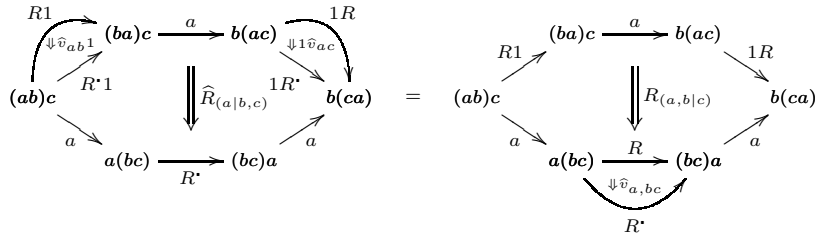
\[
\def\objectstyle{\scriptstyle}
\def\labelstyle{\scriptstyle}
\xy0;/r.20pc/:
{\ar_{R^{\dot}1} (0,0)*+{(ab)c}; (15,10)*+{(ba)c} };
{\ar@/^1.5pc/^{R1} (0,0)*+{(ab)c}; (15,10)*+{(ba)c} };
{\ar^{a} (15,10)*+{(ba)c}; (35,10)*+{b(ac)} };
{\ar_{1R^{\dot}} (35,10)*+{b(ac)}; (50,0)*+{b(ca)} };
{\ar@/^1.5pc/^{1R} (35,10)*+{b(ac)}; (50,0)*+{b(ca)} };
{\ar_{a} (0,0)*+{(ab)c}; (15,-10)*+{a(bc)} };
{\ar_{R^{\dot}} (15,-10)*+{a(bc)}; (35,-10)*+{(bc)a} };
{\ar_{a} (35,-10)*+{(bc)a}; (50,0)*+{b(ca)} };
{\ar@{=>}^{\widehat{R}_{(a|b,c)}} (25, 5)*{}; (25,-5)*{} };
(60,0)*{=}; (6,8.5)*{\scriptscriptstyle \Downarrow \widehat{v}_{ab}1};(44,8.5)*{\scriptscriptstyle \Downarrow 1\widehat{v}_{ac}};
{\ar^{R1} (70,0)*+{(ab)c}; (85,10)*+{(ba)c} };
{\ar^{a} (85,10)*+{(ba)c}; (105,10)*+{b(ac)} };
{\ar^{1R} (105,10)*+{b(ac)}; (120,0)*+{b(ca)} };
{\ar_{a} (70,0)*+{(ab)c}; (85,-10)*+{a(bc)} };
{\ar^{R} (85,-10)*+{a(bc)}; (105,-10)*+{(bc)a} };
{\ar@/_1.5pc/_{R^{\dot}} (85,-10)*+{a(bc)}; (105,-10)*+{(bc)a} };
{\ar_{a} (105,-10)*+{(bc)a}; (120,0)*+{b(ca)} };
{\ar@{=>}^{R_{(a,b|c)}} (95, 5)*{}; (95,-5)*{} };
(95,-14)*{\scriptscriptstyle \Downarrow \widehat{v}_{a,bc}}
\endxy
\]
\[
\def\objectstyle{\scriptstyle}
\def\labelstyle{\scriptstyle}
\xy0;/r.20pc/:
{\ar_{1R^{\dot}} (0,0)*+{a(bc)}; (15,10)*+{a(cb)} };
{\ar@/^1.5pc/^{1R} (0,0)*+{a(bc)}; (15,10)*+{a(cb)} };
{\ar^{a^{\centerdot}} (15,10)*+{a(cb)}; (35,10)*+{(ac)b} };
{\ar_{R^{\dot}1} (35,10)*+{(ac)b}; (50,0)*+{(ca)b} };
{\ar@/^1.5pc/^{R1} (35,10)*+{(ac)b}; (50,0)*+{(ca)b} };
{\ar_{a^{\centerdot}} (0,0)*+{a(bc)}; (15,-10)*+{(ab)c} };
{\ar_{R^{\dot}} (15,-10)*+{(ab)c}; (35,-10)*+{c(ab)} };
{\ar_{a^{\centerdot}} (35,-10)*+{c(ab)}; (50,0)*+{(ca)b} };
{\ar@{=>}^{\widehat{R}_{(a,b|c)}} (25, 7)*{}; (25,-7)*{} };
(60,0)*{=}; (6,8.5)*{\scriptscriptstyle \Downarrow 1 \widehat{v}_{bc}};(44,8.5)*{\scriptscriptstyle \Downarrow \widehat{v}_{ac}1};
{\ar^{1R} (70,0)*+{a(bc)}; (85,10)*+{a(cb)} };
{\ar^{a^{\centerdot}} (85,10)*+{a(cb)}; (105,10)*+{(ac)b} };
{\ar^{R1} (105,10)*+{(ac)b}; (120,0)*+{(ca)b} };
{\ar_{a^{\centerdot}} (70,0)*+{a(bc)}; (85,-10)*+{(ab)c} };
{\ar^{R} (85,-10)*+{(ab)c}; (105,-10)*+{c(ab)} };
{\ar@/_1.5pc/_{R^{\dot}} (85,-10)*+{(ab)c}; (105,-10)*+{c(ab)} };
{\ar_{a^{\centerdot}} (105,-10)*+{c(ab)}; (120,0)*+{(ca)b} };
{\ar@{=>}^{R_{(a|b,c)}} (95, 5)*{}; (95,-5)*{} };
(95,-14)*{\scriptscriptstyle \Downarrow \widehat{v}_{ab,c}}
\endxy
\]

A \textit{symmetric monoidal bicategory} is a sylleptic monoidal bicategory satisfying the one additional axiom below.

\noindent \textbf{Symmetry axiom:}
\[
\xy
{\ar^{R} (0,0)*+{ab}; (10,15)*+{ba} };
{\ar^{R} (10,15)*+{ba}; (30,15)*+{ab} };
{\ar^{R} (30,15)*+{ab}; (40,0)*+{ba} };
{\ar_{R} (0,0)*+{ab}; (40,0)*+{ba} };
{\ar_{1} (0,0)*+{ab}; (30,15)*+{ab} };
(13,10.5)*{\Downarrow v_{ab}}; (25,5)*{\cong};
{\ar^{R} (50,0)*+{ab}; (60,15)*+{ba} };
{\ar^{R} (60,15)*+{ba}; (80,15)*+{ab} };
{\ar^{R} (80,15)*+{ab}; (90,0)*+{ba} };
{\ar_{R} (50,0)*+{ab}; (90,0)*+{ba} };
{\ar_{1} (60,15)*+{ba}; (90,0)*+{ba} };
(65,5)*{\cong}; (77,10.5)*{\Downarrow v_{ba}};
(45,7.5)*{=}
\endxy
\]

\begin{rem}
Note that we have used slightly different notation from previous authors, notably Day-Street \cite{DS} and McCrudden \cite{mccrudden}. We have changed the direction of $v$ in order to make our rewriting argument later more natural, and we have expressed the syllepsis axiom using mates explicitly to present a more pleasing geometry.  It should be noted that this syllepsis axiom is equivalent to those of other authors, and most importantly that other versions of the syllepsis axioms do not include any instances of $R^{\dot}$; that fact will be used later.
\end{rem}

Recall that for an adjoint equivalence $f \dashv_{eq} f^{\dot}$, the unit is a 2-cell $\eta:1 \Rightarrow f^{\dot}f$ and the counit is a 2-cell $\varepsilon:ff^{\dot} \Rightarrow 1$.

\noindent \textbf{Definition.}  A symmetric monoidal bicategory $X$ is a \textit{strict symmetric monoidal bicategory} if
\begin{itemize}
\item the underlying braided monoidal bicategory of $X$ is strict as a braided monoidal bicategory, i.e., is a braided monoidal 2-category in the sense of Crans \cite{crans}, and
\item the adjoint equivalences $R_{xy} \dashv_{eq} R_{xy}^{\dot}$ is given by
\[
\begin{array}{rcl}
R_{xy}^{\dot} & = & R_{yx}, \\
\eta & = & v_{xy}^{-1}, \\
\varepsilon & = & v_{yx}.
\end{array}
\]
\end{itemize}

\begin{rem}
One might think that we should require of a strict symmetric monoidal bicategory that $v$ satisfy additional  axioms when one of the objects involved is the unit for the tensor product, but these equations are automatically satisfied as they are a byproduct of the syllepsis axioms.
\end{rem}

\noindent \textbf{Definition.}  Let $X,Y$ be symmetric monoidal bicategories.  A \textit{symmetric monoidal functor} $F:X \rightarrow Y$ is a braided monoidal functor satisfying the one additional axiom below.
\[
\def\objectstyle{\scriptstyle}
\def\labelstyle{\scriptstyle}
\xy
{\ar^{R} (0,0)*+{FaFb}; (20,0)*+{FbFa} };
{\ar^{R} (20,0)*+{FbFa}; (40,-15)*+{FaFb} };
{\ar^{\chi} (40,-15)*+{FaFb}; (40,-35)*+{F(ab)} };
{\ar_{\chi} (0,0)*+{FaFb}; (0,-20)*+{F(ab)} };
{\ar^{FR} (0,-20)*+{F(ab)}; (20,-20)*+{F(ba)} };
{\ar_{\chi} (20,0)*+{FbFa}; (20,-20)*+{F(ba)} };
{\ar^{FR} (20,-20)*+{F(ba)}; (40,-35)*+{F(ab)} };
{\ar|{F(RR)} (0,-20)*+{F(ab)}; (40,-35)*+{F(ab)} };
{\ar@/_1.5pc/_{F1} (0,-20)*+{F(ab)}; (40,-35)*+{F(ab)} };
(10,-10)*{\Downarrow U}; (30,-17.5)*{\Downarrow U}; (20,-23)*{\cong}; (20,-31)*{\Downarrow Fv};
(50,-20)*{=};
{\ar^{R} (60,0)*+{FaFb}; (80,0)*+{FbFa} };
{\ar^{R} (80,0)*+{FbFa}; (100,-15)*+{FaFb} };
{\ar^{\chi} (100,-15)*+{FaFb}; (100,-35)*+{F(ab)} };
{\ar_{\chi} (60,0)*+{FaFb}; (60,-20)*+{F(ab)} };
{\ar@/_1.5pc/_{F1} (60,-20)*+{F(ab)}; (100,-35)*+{F(ab)} };
{\ar^{1} (60,-20)*+{F(ab)}; (100,-35)*+{F(ab)} };
{\ar@/_1.5pc/_{1} (60,0)*+{FaFb}; (100,-15)*+{FaFb} };
(80,-7)*{\Downarrow v}; (80,-20)*{\cong}; (80,-31)*{\cong}
\endxy
\]

\noindent \textbf{Definition.}  A symmetric monoidal functor $F:X \rightarrow Y$ is \textit{strict} if it is strict as a braided monoidal functor (i.e., strictly preserves all of the braided monoidal structure) and has the property that $Fv_{X}=v_{Y}$.

\noindent \textbf{Definition.}  A symmetric monoidal functor $F:X \rightarrow Y$ is a \textit{symmetric monoidal biequivalence} if its underlying functor of bicategories is a biequivalence.

\begin{rem}
This definition of symmetric monoidal biequivalence is logically equivalent to requiring the existence of a symmetric monoidal pseudoinverse by results of \cite{gurski-equiv}.
\end{rem}

\noindent \textbf{Definition.}  The category $\mb{SymMonBicat}$ has objects symmetric monoidal bicategories and morphisms strict symmetric monoidal functors.

\begin{rem}
Just as the standard composition of weak functors between tricategories does not give a category of tricategories and functors, there is also not a category of monoidal bicategories and (weak) monoidal functors between.  Thus the category of symmetric monoidal bicategories must have strict functors as its morphisms, and the composition law used takes a pair of strict symmetric monoidal functors, composes the underlying functors of bicategories, and then imposes the unique strict symmetric monoidal structure on the resulting composite functor.  The corresponding discussion for tricategories can be found in the upcoming \cite{gurski-tri}.
\end{rem}

Before moving on, we draw the reader's attention to the coherence theorem for braided monoidal bicategories, found in \cite{gurski-braid} and stated below.

\begin{thm}
In the free braided monoidal bicategory on a single object, two parallel 1-cells are either uniquely isomorphic or the set of 2-cells between them is empty.  Two parallel 1-cells are isomorphic if and only if they have the same underlying braid.
\end{thm}
This theorem will be used quite heavily in our proof of coherence for the symmetric case.  We can give an informal description of the free braided monoidal bicategory on one object.  It has 0-cells which are natural numbers, and there are 1-cells $m \rightarrow n$ only when $m=n$; in that case, such a 1-cell consists of a braid on $n$ strands.  The 2-cells have a much more complicated description, being generated from the many different kinds of 2-cells necessary to give a braided monoidal structure, but the coherence theorem says that these 2-cells should be thought of as homotopy classes of homotopies between braids.  This geometric approach to coherence is discussed at length in \cite{gurski-braid}.




\subsection{Examples}

This section gives a variety of examples of naturally-occurring symmetric monoidal bicategories.  Some of these examples require a considerable amount of work in order to verify the entire symmetric monoidal structure.

\noindent \textbf{Example 1.}  In \cite{CKWW}, the authors verify that every bicategory with products, taken in the sense of bilimits, can be viewed as a symmetric monoidal bicategory.

\noindent \textbf{Example 2.} If $\mathcal{V}$ is a symmetric monoidal category, then the 2-category $\mathcal{V}$-$\mb{Cat}$ can be given a symmetric monoidal structure where the tensor product of two $\mathcal{V}$-categories is obtained by taking the product of the object sets and the tensor product (in $\mathcal{V}$) of hom-objects.  This symmetric structure is actually quite strict, as it is really the $\mb{Cat}$-enrichment of the symmetric monoidal structure on the \textit{category} of $\mathcal{V}$-categories and $\mathcal{V}$-functors.

\noindent \textbf{Example 3.}  Using the double-categorical methods of \cite{shulman}, the same tensor product as in the previous example can be used to produce a symmetric monoidal structure on the bicategory of $\mathcal{V}$-categories and $\mathcal{V}$-profunctors between them.  This includes the example of the bicategory of rings, bimodules, and bimodule homomorphisms, in which case the tensor product is given by the usual tensor product of abelian groups extended to rings (for the 0-cells) or bimodules (for the 1-cells).

\noindent \textbf{Example 4.}  Let $\mb{nCob}$ denote the bicategory with objects closed $n$-manifolds, 1-morphisms given by cobordisms, and 2-morphisms given by diffeomorphisms between cobordisms. As shown in \cite{shulman}, this bicategory is symmetric monoidal, with product given by disjoint union.

\noindent \textbf{Example 5.}  Given a bipermutative category $\mathcal{R}$, one can construct the bicategory $\mathcal{GL}(\mathcal{R})$ which is a skeletal model for the bicategory of modules over $\mathcal{R}$. In \cite{osorno}, the second author shows that $\mathcal{GL}(\mathcal{R})$ is symmetric monoidal. The group completion of the classifying space of this bicategory gives 2-K-theory, as defined in \cite{BDR}. When $\mathcal{R}$ is the category of vector spaces over $\mathbb{C}$, we get the bicategory of 2-vector spaces of Kapranov and Voevodsky \cite{KV}.

This paper adds two new classes of examples.

\noindent \textbf{New example 1.}  For our proof of coherence, we carefully study free symmetric monoidal bicategories.  These can be generated by a wide variety of data:  an underlying set of objects, an underlying bicategory, or an underlying monoidal, braided monoidal, or sylleptic monoidal bicategory.  We focus primarily on the first case, but give the construction for the free symmetric monoidal bicategory on an underlying bicategory.

\noindent \textbf{New example 2.}  For any space $X$, one can construct its fundamental 2-groupoid $\Pi_{2}X$ as in \cite{gurski-braid} or \cite{MS}.  We show that if $X$ is equipped with an action of the little $n$-cubes operad $\mathscr{C}_{n}$ for $n \geq 4$, then $\Pi_{2}X$ comes equipped with a symmetric monoidal structure.  We also show that $\Pi_{2}X$ comes equipped with a sylleptic monoidal structure when $X$ is an algebra over the little 3-cubes operad $\mathscr{C}_{3}$.  These results, combined with those in \cite{gurski-braid} give a complete description of the monoidal structure on $\Pi_{2}X$ for $X$ an algebra over any $E_{n}$ operad.

\subsection{Statements of coherence}

This section will give the statements of three different coherence theorems for symmetric monoidal bicategories.  The first is an identification of free symmetric monoidal bicategories, thus we will need to discuss the construction of such free objects first.  The second will identify the free symmetric monoidal bicategory on one object with the symmetric monoidal category of finite sets and bijections (or a skeletal version of such).  The third and final form of coherence will be a strictification theorem.

\noindent \textbf{Definition.}  Let $B$ be a bicategory.  The \textit{free symmetric monoidal bicategory} $\mathcal{S}B$ generated by $B$ has objects which are inductively constructed using the objects of $B$ and a new unit object $I$ by using a binary tensor product.  The 1-cells are inductively constructed from the 1-cells of $B$, associativity and unit 1-cells for the monoidal structure (these will be part of adjoint equivalences later, so we need 1-cells going in each direction), and braiding 1-cells $R_{xy}$ and $R_{xy}^{\dot}$; to construct new 1-cells from these, we can compose 1-cells with matching source and target or tensor 1-cells together.  The 2-cells are then inductively constructed from the 2-cells in $B$, the 2-cells for the monoidal structure ($\pi, \mu, \lambda, \rho$ and their inverses), the 2-cells for the braided structure ($R_{--|-}, R_{-|--}$ and their inverses), and the 2-cells for the symmetric structure ($v$ and its inverse), once again by composing along either 0- or 1-cell
boundaries or tensoring; the 2-cells are then quotiented by the relations forcing $\mathcal{S}B$ to be a symmetric monoidal bicategory and those which make the obvious inclusion $B \hookrightarrow \mathcal{S}B$ a strict functor of bicategories.

\noindent \textbf{Definition.} The category of bicategories and strict functors will be denoted $\mb{Bicat}_{s}$.

\begin{prop}
The assignment $B \mapsto \mathcal{S}B$ is the function on objects of a functor $\mb{Bicat}_{s} \rightarrow \mb{SymMonBicat}$ which is left adjoint to the forgetful functor.  The resulting adjunction is monadic.
\end{prop}

\begin{rem}
We are primarily interested in the case when $B$ is a set, viewed as a discrete bicategory, particularly a terminal set $*$.  In this case, we write $\mathcal{S}$ for $\mathcal{S}*$.
\end{rem}

The first coherence theorem for symmetric monoidal bicategories is the following.

\begin{thm}\label{coherencesmb}
In $\mathcal{S}$, every diagram of 2-cells commutes.  Equivalently, between every pair of parallel 1-cells there is either a unique invertible 2-cell or no 2-cells at all.  Moreover, parallel 1-cells are isomorphic if and only if they have the same underlying permutation.
\end{thm}

\begin{cor}\label{coherenceset}
For any set $X$ seen as a discrete bicategory, the free symmetric monoidal bicategory generated by $X$ has the property that every diagram of 2-cells commutes.
\end{cor}

We can extend this result even further, taking into account generating 1-cells as well as an arbitrary set of objects. This is best accomplished using a wreath product-type construction as in \cite{js}.  We record the result here although we do not give the proof.  This version of coherence is useful to know for theoretical purposes, but in practice it is often obvious how to reduce, via naturality, any question that might require such a theorem to the corresponding coherence theorem with only a set of objects.

\begin{thm}
Let $B$ be a bicategory which is free on an underlying 1-globular set.  Then in $\mathcal{S}B$ every diagram of 2-cells commutes.
\end{thm}

\noindent \textbf{Definition.}  The symmetric monoidal category $\Sigma$ has objects the natural numbers $\mb{n}$ and hom-sets given by
\[
\Sigma(\mb{m}, \mb{n}) = \left\{ \begin{array}{lr}
\Sigma_{n}, & m=n \\
\varnothing, & m \neq n,
\end{array} \right.
\]
where $\Sigma_{n}$ is the symmetric group on $n$ letters.  The tensor product is given by addition on objects and the obvious inclusion $\Sigma_{m} \times \Sigma_{n} \hookrightarrow \Sigma_{m+n}$ on morphisms.  The symmetry $c_{n,m}$ is given by the permutation in $\Sigma_{n+m}$ that takes $\{1,\dots, n+m\}$ to $\{n+1,n+2, \dots, n+m, 1, 2, \dots, n\}$. Note that $\Sigma$ is equivalent as a symmetric monoidal category to the category of finite sets and isomorphisms, with monoidal product given by disjoint union.

Our second statement of coherence is the following corollary, which is an immediate consequence of Theorem \ref{coherencesmb}.

\begin{cor}
The symmetric monoidal functor $\pi:\mathcal{S} \rightarrow \Sigma$ induced by the universal property sending the generating object of $\mathcal{S}$ to $\mb{1}$ is a symmetric monoidal biequivalence.
\end{cor}

\begin{rem}
Note that this theorem shows that, up to symmetric monoidal biequivalence, there is no difference between the free symmetric monoidal bicategory on one object and the free symmetric monoidal \textit{category} on one object (viewed as a discrete symmetric monoidal bicategory), as $\Sigma$ is a strict model for this category.
\end{rem}

The third coherence theorem for symmetric monoidal bicategories is a straightforward strictification theorem.

\begin{thm}\label{strictification}
Every symmetric monoidal bicategory is biequivalent, as a symmetric monoidal bicategory, to a strict one.
\end{thm}

We should note that this version of coherence is not so much a consequence of the previous versions as it is a complementary result.  The proof of this theorem follows more from the strategy we employ to prove the other versions of coherence rather than from those coherence theorems directly.

\section{Topological applications}

As is the case with categories, there is a classifying space construction for bicategories that allows one to transport extra structure at the level of bicategories to extra structure at the level of spaces. There is also a fundamental 2-groupoid functor from spaces to bigroupoids, constructed in \cite{gurski-braid} or \cite{MS}. It carries all the information about the homotopy 2-type of the space.  In this section we explore how the symmetric monoidal structure interacts with these two functors. It is no surprise that the coherence theorem (Theorem \ref{coherencesmb}) implies that the classifying space of a symmetric monoidal bicategory is an $E_{\infty}$ space, which we prove in $\S$\ref{gamma}. In $\S$\ref{pi}, we prove that the fundamental 2-groupoid of an $E_n$ space has the structure of a symmetric monoidal bicategory, if $n\geq 4$, and sylleptic monoidal bicategory if $n=3$.

\subsection{Symmetric monoidal bicategories and $\Gamma$-spaces}\label{gamma}

In this section we prove that the classifying space of a symmetric monoidal category is an $E_{\infty}$ space. In order to do this, we use Segal's theory of $\Gamma$-spaces \cite{segal}. This improves on the results of the second author in \cite{osorno}, where a similar result is proved for a collection of symmetric monoidal bicategories satisfying some extra hypotheses, which were sufficient for the applications in mind\footnote{In \cite{osorno}, the second author chose the term \emph{strict} to refer to the symmetric monoidal bicategories in that collection. We would like to point out that the strict symmetric monoidal bicategories in that sense are not necessarily strict as defined in this paper, nor conversely. Thus the strictification theorem (Theorem \ref{strictification}) does not produce symmetric monoidal bicategories that satisfy the axioms in \cite{osorno}.}.

We  denote by $\cF$ the skeletal category of finite pointed sets and pointed maps, with objects given by $\ul{n}=\{ 0, \dots, n\}$, where 0 is the basepoint. This category is isomorphic to the opposite of the category $\Gamma$ defined by Segal.

A $\Gamma$-\emph{space} is a functor $X\colon \cF \to \mb{Top} _{\ast}$, from $\cF$ to pointed spaces. It is said to be \emph{special} if the structural map
\[
p_n\colon X(\ul{n}) \to X(\ul{1})^n
\]
is a weak equivalence for all $n\geq0$. We then say that $X(\ul{1})$ is an $E_{\infty}$ space.

\begin{thm}\cite[Prop. 1.4]{segal}\label{gspaces}
Let $X$ be a special $\Gamma$-space. Then $X(\ul{1})$ is an infinite loop space upon group completion.
\end{thm}

The following definition is analogous to that for spaces. It was first given in \cite{osorno}.

\noindent \textbf{Definition.}  A \emph{$\Gamma$-bicategory} $\A$ is a functor (of categories) $\A\colon \cF\rightarrow \textrm{Bicat}_{\ast}$ from $\cF$ to the category of pointed bicategories and pointed pseudofunctors between them. We say $\A$ is \emph{special} if the map
$$p_n\colon \A(\ul{n})\rightarrow \A(\ul{1})^{\times n}$$
is a biequivalence of bicategories for all $n\geq 0$.

\begin{lem}\cite{osorno}\label{lemma}
 Let $\A$ be a special $\Gamma$-bicategory. Then $|\bN \A|\colon\cF \rightarrow \mb{Top}_{\ast}$ is a special $\Gamma$-space.
\end{lem}

Here, $|\bN X|$ denotes the geometric realization of the nerve of the bicategory $X$; we also refer to this space as the classifying space of $X$.

Instead of constructing directly a special $\Gamma$-bicategory, we will construct a weakened version and then use a rectification theorem of \cite{ccg-lax}.


\noindent \textbf{Definition.} A \emph{pseudo-diagram of bicategories indexed by} $\cF$ is a pseudofunctor from $\cF$ into the tricategory $\mb{Bicat}$ of bicategories, pseudonatural transformations, and modifications.

The data for such a functor is written out explicitly in \cite[\S 2.2]{ccg-lax}. In what follows, we will use their notation.

\begin{thm}
Let $\cC$ be a symmetric monoidal bicategory. There exists a pseudo-diagram of bicategories $\gc$ with level $n$ given by $\cC ^n$. 
\end{thm}

\begin{proof}
We construct $\gc$ as follows. Let $\theta \colon \ul{n} \to \ul{m}$ be a morphism in $\cF$. The  homomorphism $\theta _{\ast} \colon \cC ^n \to \cC ^m$ sends the $n$-tuple $(a_1, \dots, a_n)$ to the $m$-tuple $(b_1, \dots, b_m)$, where
\[b_i=\bigoplus_{j\in \theta^{-1}(i)}a_j\]
with the convention that the sum is taken in order, a parenthesization has been chosen --for definiteness, say from left to right--, and the empty sum returns the identity object of the symmetric monoidal bicategory. The same is done with 1-cells and 2-cells. Note that this is an iteration of the monoidal product in $\cC$, hence it is a homomorphism of bicategories.

Given composable morphisms $\ul{n} \fto{\theta} \ul{m}$, $\ul{m}\fto{\tau}\ul{p}$ in $\cF$, note that $(\tau_{\ast} \theta_{\ast})(a_1, \dots, a_n)$ and $(\tau \theta)_{\ast} (a_1, \dots, a_n)$ are sums of exactly the same terms of the $n$-tuple, but most likely added in different order and with different parenthesization. Thus there exists a pseudonatural transformation $\chi _{\theta, \tau}\colon \tau_{\ast} \theta_{\ast} \to (\tau \theta)_{\ast} $ constructed as a composite of instances of the pseudonatural transformations $\mathbf{a}$, $\mathbf{l}$, $\mathbf{r}$ and $R$ which are part of the structure of the symmetric monoidal bicategory. We remark that this choice is not necessarily unique, but one such exists, and we choose one for every pair $(\tau, \sigma)$.

On the other hand, note that $(1_n)_{\ast}$ is the identity homomorphism on $\cC^n$, thus we can take the pseudonatural transformation $\iota _{n}$ to be the identity.

The remaining data needed are the modifications $\omega$, $\delta$ and $\gamma$:

\[
\def\labelstyle{\scriptstyle}
\xy
{\ar^{{\phi _{\ast} \chi}} (0,0)*+{\phi_{\ast} \tau _{\ast} \theta _{\ast}}; (20,0)*+{\phi_{\ast}( \tau  \theta) _{\ast}} };
{\ar_{\chi \theta_{\ast}} (0,0)*+{\phi_{\ast} \tau _{\ast} \theta _{\ast}}; (0,-20)*+{ (\phi \tau )_{\ast}\theta _{\ast}} };
{\ar_{\chi} (0,-20)*+{ (\phi \tau )_{\ast}\theta _{\ast}}; (20,-20)*+{(\phi \tau  \theta) _{\ast}} };
{\ar^{\chi} (20,0)*+{\phi_{\ast}( \tau  \theta) _{\ast}}; (20,-20)*+{(\phi \tau  \theta) _{\ast}} };
(10,-10)*{\Rightarrow _{\omega}};
{\ar^{{\theta _{\ast} \iota}} (40,0)*+{\theta _{\ast}}; (70,0)*+{\theta _{\ast} (1_n)_{\ast}} };
{\ar_{\iota \theta_{\ast}} (40,0)*+{\theta _{\ast}}; (40,-20)*+{ (1_m)_{\ast}\theta _{\ast}} };
{\ar_{\chi} (40,-20)*+{ (1_m)_{\ast}\theta _{\ast}}; (70,-20)*+{\theta _{\ast}} };
{\ar^{\chi} (70,0)*+{\theta _{\ast} (1_n)_{\ast}}; (70,-20)*+{\theta _{\ast}} };
{\ar^{1} (40,0)*+{\theta _{\ast} }; (70,-20)*+{\theta _{\ast}} };
(10,-10)*{\Rightarrow _{\omega}}; (62,-7)*{\Rightarrow _{\delta}}; (50,-13)*{\Rightarrow _{\gamma}};

\endxy
\]

Note that the two legs of the diagram on the left are composites of 1-cells coming from the symmetric monoidal structure on $\cC$, with the same underlying permutation. Thus, by Theorem \ref{coherencesmb}, there exists a unique 2-isomorphism between them. We let $\omega$ be this unique isomorphism. The commutativity of the diagram \cite[(CC1)]{ccg-lax} follows from this uniqueness. A similar argument works for the construction of the modifications $\delta$ and $\gamma$ above; we define $\delta$ and $\gamma$ as the unique 2-isomorphisms for the symmetric monoidal structure that fill the diagram on the right. These 2-isomorphims exist because of Theorem \ref{coherencesmb}. Diagram \cite[(CC2)]{ccg-lax} commutes because of the uniqueness part of the theorem.
\end{proof}

\begin{rem}
Note that Theorem \ref{coherencesmb} is used in the above proof in an essential way; the coherence theorem for symmetric monoidal bicategories implies the coherence of the data for the pseudo-diagram of bicategories.
\end{rem}

The rectification constructed in \cite[Prop 4.1]{ccg-lax} provides a strict diagram $\gc ^r$. Note that the map $\ast=\cC^0\fto{J_0} \gc^r(0)\to \gc^r(n)$ provides the bicategory $\gc^r(n)$ with a basepoint. Furthermore, the strict functoriality of $\gc^r$ implies that for all $\theta\colon \ul{n}\to\ul{m}$ in $\cF$, the pseudofunctor $\gc^r(\theta)$ is pointed. We thus get that $\gc^r$ is a $\Gamma$-bicategory. As a consequence of \cite[Prop 4.2]{ccg-lax}, the map $J_n:\cC^n=\gc(n) \to \gc^r(n)$ is an equivalence of bicategories, thus showing that $\gc^ r$ is a special $\Gamma$-bicategory, with $\gc ^r(1)$ equivalent to $\cC$.

\begin{cor}
The classifying space $|\bN \gc|$ is a special $\Gamma$-space. The space $|\bN \cC|$ is an $E_{\infty}$ space.
\end{cor}

\begin{rem}
Note that this corollary does not depend on a specific choice of classifying space functor from bicategories to spaces, only that this functor preserves products up to weak equivalence.
\end{rem}

\subsection{Fundamental 2-groupoids of $E_{n}$ spaces}\label{pi}

In this section we prove that the fundamental 2-groupoid of a $E_n$ space for $n=3$, and $n\geq 4$ can be equipped with the structure of a sylleptic, respectively symmetric, monoidal bicategory. The strategy will be similar to the one presented by the first author in \cite{gurski-braid}, where he proves that $E_1$ structures give rise to monoidal structures, and $E_2$ algebras give rise to braided monoidal structures.

Let $X$ be a space. The \emph{fundamental 2-groupoid} of $X$ is the bigroupoid $\Pi_2X$ with objects given by the points of $X$, 1-cells given by paths in $X$ and 2-cells given by homotopy classes of homotopies. The details of the construction can be found in \cite[\S 3.2]{gurski-braid}.

Recall that an operad $\sP$ in a symmetric monoidal category $(\cM, \otimes)$ is a device used to describe some notion of algebraic structure on the objects of $\cM$. The operad $\sP$ is given by a collection of objects $\sP(n)$ in $\cM$ for all $n\geq 0$, and structure maps
\[
\sP(n)\otimes \sP(k_1) \otimes \cdots \otimes \sP(k_n)\to \sP(k_1+\cdots + k_n),
\]
a unit map $I\to \sP(1)$, and right $\Sigma _n$-action on $\sP(n)$, making the maps above appropriately equivariant. These maps are required to satisfy unit and associativity axioms.  We are only interested in the case where $\cM=\mb{Top}$ and $\sP$ is the little $n$-cubes operad $\sC_n$, where $\sC_n$ is either 3 or 4.

The little $n$-cubes operad was introduced by Bordman and Vogt (as a Prop) \cite{BV} and May \cite{maygeom} to encode $E_n$ structures on spaces.  Let $J$ denote the open unit interval.  A \textit{little n-cube} is a linear embedding $\alpha:J^{n} \rightarrow J^{n}$ which is of the form $\alpha = \alpha_{1} \times \alpha_{2} \times \cdots \times \alpha_{n}$ where each $\alpha_{i}$ is a linear map
\[
\alpha_{i}(t) = (y_{i} - x_{i})t + x_{i}, \quad 0 \leq x_{i} < y_{i} \leq 1.
\]
The space $\mathcal{C}_{n}(k)$ is the subspace of $\textrm{Map} \big( (J^{n})^{k}, J^{n} \big)$ consisting of those $k$-tuples of little $n$-cubes which are pairwise disjoint.  There is an operadic multiplication on the spaces of little $n$-cubes given by composition of maps; geometrically, this corresponds to embedding cubes into other cubes.  It is then simple to check that this is a operad in the category of spaces. Here is an example of a point in $\sC_2(2)\times \sC_2(3) \times \sC_2(1)$:
 
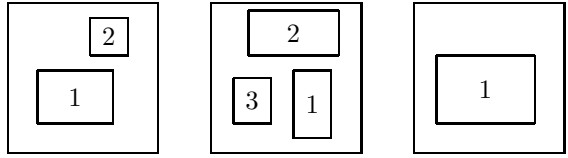
\[
\begin{xy}
(0,0);(20,0) **\dir{-};
(0,0);(0,20) **\dir{-};
(20,0);(20,20) **\dir{-};
(0,20);(20,20) **\dir{-};
(4,4);(14,4) **\dir{-};
(4,4);(4,11) **\dir{-};
(14,4);(14,11) **\dir{-};
(4,11);(14,11) **\dir{-};
(11,13);(16,13) **\dir{-};
(11,13);(11,18) **\dir{-};
(16,13);(16,18) **\dir{-};
(11,18);(16,18) **\dir{-};
(13.5,15.5)*{2};
(9,7.5)*{1};
\end{xy}
\quad  \quad
\begin{xy}
(0,0);(20,0) **\dir{-};
(0,0);(0,20) **\dir{-};
(20,0);(20,20) **\dir{-};
(0,20);(20,20) **\dir{-};
(3,4);(8,4) **\dir{-};
(3,4);(3,10) **\dir{-};
(8,4);(8,10) **\dir{-};
(3,10);(8,10) **\dir{-};
(5,13);(5,19) **\dir{-};
(5,13);(17,13) **\dir{-};
(5,19);(17,19) **\dir{-};
(17,13);(17,19) **\dir{-};
(11,2);(16,2) **\dir{-};
(11,2);(11,11) **\dir{-};
(16,2);(16,11) **\dir{-};
(11,11);(16,11) **\dir{-};
(13.5,6.5)*{1};
(11,16)*{2};
(5.5,7)*{3};
\end{xy}
\quad  \quad
\begin{xy}
(0,0);(20,0) **\dir{-};
(0,0);(0,20) **\dir{-};
(20,0);(20,20) **\dir{-};
(0,20);(20,20) **\dir{-};
(3,4);(16,4) **\dir{-};
(3,4);(3,13) **\dir{-};
(16,4);(16,13) **\dir{-};
(3,13);(16,13) **\dir{-};
(9.5,8.5)*{1};
\end{xy}
\]
The image of this point under the multiplication map is the following point in $\sC_2(4)$.
\[
\begin{xy}
(0,0);(30,0) **\dir{-};
(0,0);(0,30) **\dir{-};
(30,0);(30,30) **\dir{-};
(0,30);(30,30) **\dir{-};
(6,6);(21,6) **\dir{.};
(6,6);(6,16.5) **\dir{.};
(21,6);(21,16.5) **\dir{.};
(6,16.5);(21,16.5) **\dir{.};
(16.5,19.5);(24,19.5) **\dir{.};
(16.5,19.5);(16.5,27) **\dir{.};
(24,19.5);(24,27) **\dir{.};
(16.5,27);(24,27) **\dir{.};
(8.25,8.1);(12,8.1) **\dir{-};
(8.25,8.1);(8.25,11.25) **\dir{-};
(12,8.1);(12,11.25) **\dir{-};
(8.25,11.25);(12,11.25) **\dir{-};
(9.75,12.825);(9.75,15.975) **\dir{-};
(9.75,12.825);(18.75,12.825) **\dir{-};
(9.75,15.975);(18.75,15.975) **\dir{-};
(18.75,12.825);(18.75,15.975) **\dir{-};
(14.25,7.05);(18,7.05) **\dir{-};
(14.25,7.05);(14.25,11.775) **\dir{-};
(18,7.05);(18,11.775) **\dir{-};
(14.25,11.775);(18,11.775) **\dir{-};
(17.625,21);(22.5,21) **\dir{-};
(17.625,21);(17.625,24.375) **\dir{-};
(22.5,21);(22.5,24.375) **\dir{-};
(17.625,24.375);(22.5,24.375) **\dir{-};
(10.125,9.675)*{3};
(14.25,14.4)*{2};
(16.125,9.4125)*{1};
(20.0625,22.6875)*{4};
\end{xy}
\]

\begin{lem}\cite[Lemma 13]{gurski-braid}\label{operad}
Let $\mathscr{P}$ be an operad in $Top$ and $X$ a $\mathscr{P}$-algebra. Then every point $p\in \mathscr{P}(n)$ induces a map $\mu _p\colon X^n \to X$, every path $\gamma\colon I \to \mathscr{P}(n)$ gives a homotopy
\[\widetilde{\gamma}\colon \mu _{\gamma(0)} \Rightarrow \mu_{\gamma(1)},\]
and every map $\Phi \colon I^2 \to \mathscr{P}(n)$ induces a homotopy between homotopies
\[\widetilde{\Phi}\colon \widetilde{\Phi(-,0)}\to \widetilde{\Phi(-,1)}.\]
\end{lem}

\begin{thm}\label{en}
Let $X$ be an algebra over the operad $\sC_n$. Then $\Pi_2 X$ has the structure of a
\begin{enumerate}
 \item[(a)] sylleptic monoidal bicategory if $n=3$,
 \item[(b)] symmetric monoidal bicategory if $n\geq 4$.
\end{enumerate}
\end{thm}

\begin{proof}[Proof of Theorem \ref{en}(b)]
 The proof will follow the same argument of the proof of \cite[Theorem 15]{gurski-braid}. Indeed, the construction of the braided monoidal structure is exactly the same, with the first two coordinates of the centers of the little 4-cubes being the same as for the 2-cubes in \cite{gurski-braid}, and the two extra coordinates being constant at $\frac{1}{2}$. For example, the multiplication is given by the functor
\[
\Pi_2 X \times \Pi_2 X \cong \Pi_2(X\times X) \fto{\Pi _ 2 \mu_m} \Pi_2 X,
\]
where $m\in \sC_4(2)$ is the point given by the two 4-cubes with side length $\dfrac{1}{5}$, and centers at
\[
\biggl(\dfrac{3}{10}, \hf, \hf, \hf\biggr) \quad \text{and} \quad \biggl(\dfrac{7}{10}, \hf, \hf, \hf\biggr).
\]

From now on we will ignore the sizes of the cubes, as they can be made small enough to ensure that the little cubes do not intersect as long as the centers are not equal. The remaining data needed is the syllepsis $v$. This is a homotopy between the path given by $R_{yx}R_{xy}$ and the constant path at $xy$, or equivalently a nullhomotopy for $R_{yx}R_{xy}$. By Lemma \ref{operad}, this can be constructed as a map $I^2 \to \sC_4(2)$, but we will instead construct is as a map $v\colon D^2 \to \sC_4(2)$ whose boundary is given by the path $R_{yx}R_{xy}$. We give the value of the map on the centers of the 4-cubes as follows, where we are using polar coordinates $(r,t)$ to parametrize the disc:
\begin{align*}
v_1(t,r)&=\biggl( \hf + \dfrac{r}{5}\cos(\pi +2\pi t), \hf + \dfrac{r}{5}\sin(\pi +2\pi t), \hf +\dfrac{1}{5}\sqrt{1-r^2},\hf\biggr)\\
v_2(r,t)&=\biggl( \hf + \dfrac{r}{5}\cos(2\pi t), \hf + \dfrac{r}{5}\sin(2\pi t), \hf -\dfrac{1}{5}\sqrt{1-r^2},\hf\biggr).
\end{align*}

Note that the third coordinate is necessary to ensure that for all $0\leq r \leq 1$, $0\leq t<2\pi$, the points $v_1(r,t)$ and $v_2(r,y)$ are different, and thus we get a point in $\sC_4(2)$.

It remains to prove that the syllepsis and the symmetry axioms hold. All the 2-cells in question have been constructed as maps $D^2\to \sC_4(k)$ for some $k\geq0$. Since $\pi_2(\sC_4(k))=0$ for all $k\geq 0$, all diagrams involving 2-cells constructed in this fashion will commute, in particular, the axioms for a symmetric monoidal bicategory will hold automatically.

If $X$ is an algebra over $\sC_n$ for $n>4$, we use the map of operads $\sC_4 \rightarrow \sC_n$ given by sending a little 4-cube $\alpha$ to $\alpha \times J^{n-4}$ to induce an algebra structure over $\sC_4$.  We then apply the above construction.

\end{proof}

\begin{rem}
The reader might have noticed the discrepancy between the results of \S\ref{gamma}, where symmetric monoidal implied $E_{\infty}$, and the results of this section, where $E_4$ is enough to guarantee the symmetric monoidal structure. This difference is explained by the fact that the fundamental 2-groupoid of a space depends only on its homotopy 2-type. The obstructions to lifting an $E_4$ structure to an $E_{\infty}$ structure live in higher homotopy groups, thus are not seen by the fundamental 2-groupoid.
\end{rem}

The proof given above of the symmetric structure on $\Pi_{2}X$ uses an explicit construction of the homotopy $v:R^{2} \Rightarrow 1$.  We could instead construct its mate, a homotopy $\widehat{v}:R \Rightarrow R^{\dot}$, and proceed with the rest of the proof using this 2-cell.  (Note that $\widehat{v}$ corresponds to the 2-cell $t^{-1}$ used in $\S$3.1, but since $t$ is traditionally used as the time variable when expressing homotopies, we prefer to denote this 2-cell as $\widehat{v}$ in this context.)  While checking the symmetry axioms would proceed as before using the fact that $\pi_{2}(\sC_4(k))=0$ for any value of $k$, we can use this construction to study the fundamental 2-groupoid of an $E_{3}$ space as well.

\begin{rem}
Note that the proof of Theorem \ref{en}(a) is not quite as easy as the symmetric case, since the relevant homotopy groups of
$\sC_3(k)$ are nonzero; see \cite{fn} for computations of these homotopy groups.  The key point is then showing that the syllepsis axioms hold, which reduces to constructing a homotopy between a pair of maps $I^{2} \rightarrow \sC_3(3)$.

In fact, this should be a completely general phenomenon.  Assuming an explicit definition of symmetric monoidal $n$-category and of fundamental $n$-groupoids, showing that $\Pi_{n}X$, where $X$ is a space with an $E_{n+k}$ structure for $k \geq 2$, amounts merely to constructing data, as the axioms will all reduce to computations in homotopy groups which are known to be zero.  The case $k=1$ is what we are considering here for $n=2$, and for larger values of $n$ a similar strategy must be followed:  the top-dimensional data will be given as maps $I^{n} \rightarrow \sC_{n+1}(j)$, and the axioms will be equivalent to showing that certain elements of $\pi_{n}(\sC_{n+1}(j)($ are zero even though this is a nonzero homotopy group.  Hence such a proof requires explicit knowledge of the maps $I^{n} \rightarrow \sC_{n+1}(j)$.
\end{rem}

\begin{proof}[Proof of Theorem \ref{en}(a)]
Recall from \cite{gurski-braid} that the 1-cell $R_{a,b}$ is defined as the following path in $\sC_3(2)$. 
\begin{align*}
 R_a(t)&=\biggl( \hf + \dfrac{1}{5}\cos(\pi + \pi t),\hf +\dfrac{1}{5}\sin(\pi + \pi t)\biggr)\\
 R_b(t)&=\biggl( \hf + \dfrac{1}{5}\cos(\pi t),\hf +\dfrac{1}{5}\sin(\pi t)\biggr),
\end{align*}
where each function defines the center of a little 3-cube, so taken as a pair give values in $\sC_3(2)$; the subscripts $a,b$ are used to label the little 3-cubes. As above, we will ignore the lengths of the sides of the cubes.

We first define the map $\widehat{v}:I^2 \to \sC_3(2)$, so that the boundaries are as in the figure below.  Note that that top boundary is the path that gives the 1-cell $R$, while the bottom boundary is the pseudo-inverse $R^{\dot}$, which is merely $R$ run backwards.
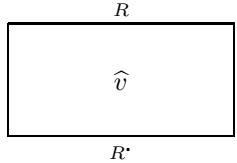
\[
\xy
{\ar@{-}^{R} (0,0)*{}; (30,0)*{} };
{\ar@{-}_{R^{\dot}} (0,-15)*{}; (30,-15)*{} };
{\ar@{-} (0,0)*{}; (0,-15)*{} };
{\ar@{-} (30,0)*{}; (30,-15)*{} };
(15,-7.5)*{\widehat{v}}
\endxy
\]
We remind the reader that we use the conventions of \cite{gurski-braid} in which a homotopy between two paths is presented as a function with domain $I^{2}$.  The variable $t$ is that of the functions $f,g$ that we are constructing a homotopy between, and runs horizontally, while the variable $s$ is the variable of the homotopy itself that runs from top to bottom.
\begin{align*}
\widehat{v}_a(t,s)&=\biggl( \hf+\dfrac{1}{5}\cos(\pi +\pi t),\hf +\dfrac{1}{5}(1-2s)\sin(\pi+\pi t), \hf +\dfrac{1}{4}\sin(\pi s)\biggr)\\
\widehat{v}_b(t,s)&=\biggl( \hf+\dfrac{1}{5}\cos(\pi t),\hf +\dfrac{1}{5}(1-2s)\sin(\pi t), \hf -\dfrac{1}{4}\sin(\pi s)\biggr)
\end{align*}

We need to prove that this $\widehat{v}$ satisfies the syllepsis axioms. We will prove one of them; the proof for the other axiom works the same, as the axioms are symmetric. Proving that the axiom holds is equivalent to constructing a homotopy between the following two maps $[0,3]\times [0,3]\to \sC_3(3)$.
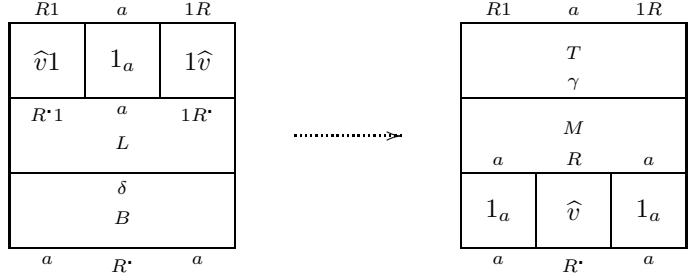
\[
\xy
{\ar@{-}^{R1} (0,0)*{}; (10,0)*{} };
{\ar@{-}^{a} (10,0)*{}; (20,0)*{} };
{\ar@{-}^{1R} (20,0)*{}; (30,0)*{} };
{\ar@{-} (0,0)*{}; (0,-30)*{} };
{\ar@{-} (30,0)*{}; (30,-30)*{} };
{\ar@{-}_{R^{\dot}1} (0,-10)*{}; (10,-10)*{} };
{\ar@{-}_{a} (10,-10)*{}; (20,-10)*{} };
{\ar@{-}_{1R^{\dot}} (20,-10)*{}; (30,-10)*{} };
{\ar@{-} (10,0)*{}; (10,-10)*{} };
{\ar@{-} (20,0)*{}; (20,-10)*{} };
{\ar@{-}_{\delta} (0,-20)*{}; (30,-20)*{} };
{\ar@{-}_{a} (0,-30)*{}; (10,-30)*{} };
{\ar@{-}_{R^{\dot}} (10,-30)*{}; (20,-30)*{} };
{\ar@{-}_{a} (20,-30)*{}; (30,-30)*{} };
{\ar@{.>} (38,-15)*{}; (52,-15)*{} };
{\ar@{-}^{R1} (60,0)*{}; (70,0)*{} };
{\ar@{-}^{a} (70,0)*{}; (80,0)*{} };
{\ar@{-}^{1R} (80,0)*{}; (90,0)*{} };
{\ar@{-}^{\gamma} (60,-10)*{}; (90,-10)*{} };
{\ar@{-}_{a} (60,-30)*{}; (70,-30)*{} };
{\ar@{-}_{R^{\dot}} (70,-30)*{}; (80,-30)*{} };
{\ar@{-}_{a} (80,-30)*{}; (90,-30)*{} };
{\ar@{-}^{a} (60,-20)*{}; (70,-20)*{} };
{\ar@{-}^{R} (70,-20)*{}; (80,-20)*{} };
{\ar@{-}^{a} (80,-20)*{}; (90,-20)*{} };
{\ar@{-} (60,0)*{}; (60,-30)*{} };
{\ar@{-} (90,0)*{}; (90,-30)*{} };
{\ar@{-} (70,-20)*{}; (70,-30)*{} };
{\ar@{-} (80,-20)*{}; (80,-30)*{} };
(5,-5)*{\widehat{v}1}; (15,-5)*{1_{a}}; (25,-5)*{1\widehat{v}};
(65,-25)*{1_{a}}; (75,-25)*{\widehat{v}}; (85,-25)*{1_{a}}; (15,-16)*{\scriptstyle L}; (15,-26)*{\scriptstyle B}; (75,-4)*{\scriptstyle T}; (75,-14)*{\scriptstyle M}
\endxy
\]

The intermediate step $\delta$ in the source above is the path $[0,3]\to \sC_3(3)$ given by
\begin{align*}
\delta_a(t)&=\biggl( \hf + \dfrac{6}{25} \cos (\pi + \dfrac{\pi}{3}t),\hf + \dfrac{6}{25}\sin(\pi -\dfrac{\pi}{3}t),\hf\biggr)\\
\delta_b(t)&=\biggl( \dfrac{8}{25} + \dfrac{1}{50} \cos (\dfrac{\pi}{3}t),\hf + \dfrac{1}{50}\sin(-\dfrac{\pi}{3}t),\hf\biggr)\\
\delta_c(t)&=\biggl( \dfrac{17}{25} + \dfrac{1}{50} \cos (\dfrac{\pi}{3}t),\hf + \dfrac{1}{50}\sin(-\dfrac{\pi}{3}t),\hf\biggr).
\end{align*}
On the other hand, the intermediate step $\gamma$ in the target diagram is given by the path $[0,3]\to \sC_3(3)$:
\begin{align*}
\gamma_a(t)&=\biggl( \hf + \dfrac{6}{25} \cos (\pi + \dfrac{\pi}{3}t),\hf + \dfrac{6}{25}\sin(\pi +\dfrac{\pi}{3}t),\hf\biggr)\\
\gamma_b(t)&=\biggl( \dfrac{8}{25} + \dfrac{1}{50} \cos (\dfrac{\pi}{3}t),\hf + \dfrac{1}{50}\sin(\dfrac{\pi}{3}t),\hf\biggr)\\
\gamma_c(t)&=\biggl( \dfrac{17}{25} + \dfrac{1}{50} \cos (\dfrac{\pi}{3}t),\hf + \dfrac{1}{50}\sin(\dfrac{\pi}{3}t),\hf\biggr).
\end{align*}

The homotopies $L,B,T,M$ from the diagram are the obvious linear homotopies, which remain constant in the third coordinate. Note that $T$ and $M$ can be constructed so that the following equations hold, where $L_i$ and $B_i$ denote the $i$th coordinate of the functions $L$ and $B$ respectively.
\begin{align*}
 T(t,s)&=\biggl(L_1(t,s-1), 1-L_2(t, s-1), \hf\biggr)\\
 M(t,s)&=\biggl(B_1(t,s-1),1-B_2(t,s-1),\hf\biggr)
\end{align*}

To construct the homotopy we need, we construct homotopies to and from an intermediate step, as shown in the figure below.
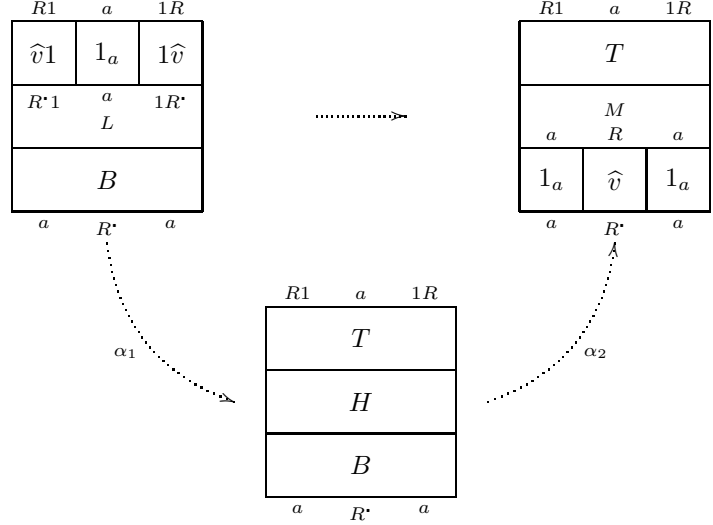
\[
\xy 0;/r.20pc/:
{\ar@{-}^{R1} (0,0)*{}; (10,0)*{} };
{\ar@{-}^{a} (10,0)*{}; (20,0)*{} };
{\ar@{-}^{1R} (20,0)*{}; (30,0)*{} };
{\ar@{-} (0,0)*{}; (0,-30)*{} };
{\ar@{-} (30,0)*{}; (30,-30)*{} };
{\ar@{-}_{R^{\dot}1} (0,-10)*{}; (10,-10)*{} };
{\ar@{-}_{a} (10,-10)*{}; (20,-10)*{} };
{\ar@{-}_{1R^{\dot}} (20,-10)*{}; (30,-10)*{} };
{\ar@{-} (10,0)*{}; (10,-10)*{} };
{\ar@{-} (20,0)*{}; (20,-10)*{} };
{\ar@{-} (0,-20)*{}; (30,-20)*{} };
{\ar@{-}_{a} (0,-30)*{}; (10,-30)*{} };
{\ar@{-}_{R^{\dot}} (10,-30)*{}; (20,-30)*{} };
{\ar@{-}_{a} (20,-30)*{}; (30,-30)*{} };
{\ar@{.>} (48,-15)*{}; (62,-15)*{} };
{\ar@{-}^{R1} (80,0)*{}; (90,0)*{} };
{\ar@{-}^{a} (90,0)*{}; (100,0)*{} };
{\ar@{-}^{1R} (100,0)*{}; (110,0)*{} };
{\ar@{-} (80,-10)*{}; (110,-10)*{} };
{\ar@{-}_{a} (80,-30)*{}; (90,-30)*{} };
{\ar@{-}_{R^{\dot}} (90,-30)*{}; (100,-30)*{} };
{\ar@{-}_{a} (100,-30)*{}; (110,-30)*{} };
{\ar@{-}^{a} (80,-20)*{}; (90,-20)*{} };
{\ar@{-}^{R} (90,-20)*{}; (100,-20)*{} };
{\ar@{-}^{a} (100,-20)*{}; (110,-20)*{} };
{\ar@{-} (80,0)*{}; (80,-30)*{} };
{\ar@{-} (110,0)*{}; (110,-30)*{} };
{\ar@{-} (90,-20)*{}; (90,-30)*{} };
{\ar@{-} (100,-20)*{}; (100,-30)*{} };
(5,-5)*{\widehat{v}1}; (15,-5)*{1_{a}}; (25,-5)*{1\widehat{v}};
(85,-25)*{1_{a}}; (95,-25)*{\widehat{v}}; (105,-25)*{1_{a}};
{\ar@{-}^{R1} (40,-45)*{}; (50,-45)*{} };
{\ar@{-}^{a} (50,-45)*{}; (60,-45)*{} };
{\ar@{-}^{1R} (60,-45)*{}; (70,-45)*{} };
{\ar@{-}_{a} (40,-75)*{}; (50,-75)*{} };
{\ar@{-}_{R^{\dot}} (50,-75)*{}; (60,-75)*{} };
{\ar@{-}_{a} (60,-75)*{}; (70,-75)*{} };
{\ar@{-} (40,-45)*{}; (40,-75)*{} };
{\ar@{-} (70,-45)*{}; (70,-75)*{} };
{\ar@{-} (40,-55)*{}; (70,-55)*{} };
{\ar@{-} (40,-65)*{}; (70,-65)*{} };
(55,-50)*{T}; (55,-70)*{B}; (55,-60)*{H};
(15,-25)*{B}; (95,-5)*{T};
{\ar@/_1pc/@{.>}_{\alpha_{1}} (15,-35)*{}; (35,-60)*{} };
{\ar@/_1pc/@{.>}_{\alpha_{2}} (75,-60)*{}; (95,-35)*{} };
(15,-16)*{\scriptstyle L};  (95,-14)*{\scriptstyle M}
\endxy
\]

The homotopy $H$ from $\gamma$ to $\delta$ is the map $[0,3]\times[1,2]\to \sC_3(3)$ defined as
\begin{align*}
H_a(t)&=\biggl( \hf + \dfrac{6}{25} \cos (\pi + \dfrac{\pi}{3}t),\hf + \dfrac{6}{25}(3-2s)\sin(\pi +\dfrac{\pi}{3}t),\hf+\dfrac{1}{4}\sin(\pi(s-1))\biggr)\\
H_b(t)&=\biggl( \dfrac{8}{25} + \dfrac{1}{50} \cos (\dfrac{\pi}{3}t),\hf + \dfrac{1}{50}(3-2s)\sin(\dfrac{\pi}{3}t),\hf-\dfrac{1}{4}\sin(\pi(s-1))\biggr)\\
H_c(t)&=\biggl( \dfrac{17}{25} + \dfrac{1}{50} \cos (\dfrac{\pi}{3}t),\hf + \dfrac{1}{50}(3-2s)\sin(\dfrac{\pi}{3}t),\hf-\dfrac{1}{4}\sin(\pi(s-1))\biggr).
\end{align*}

Note that the homotopy $(1\widehat{v})a(\widehat{v}1)$ followed by $L$ first switches the sign of the second coordinate by using the third coordinate, and then smooths the paths linearly. On the other hand, $T$ followed by $H$ first linearizes and then switches the signs. Thus, in order to construct the homotopy $\alpha_1$, we will construct a further intermediate step as in the figure below.

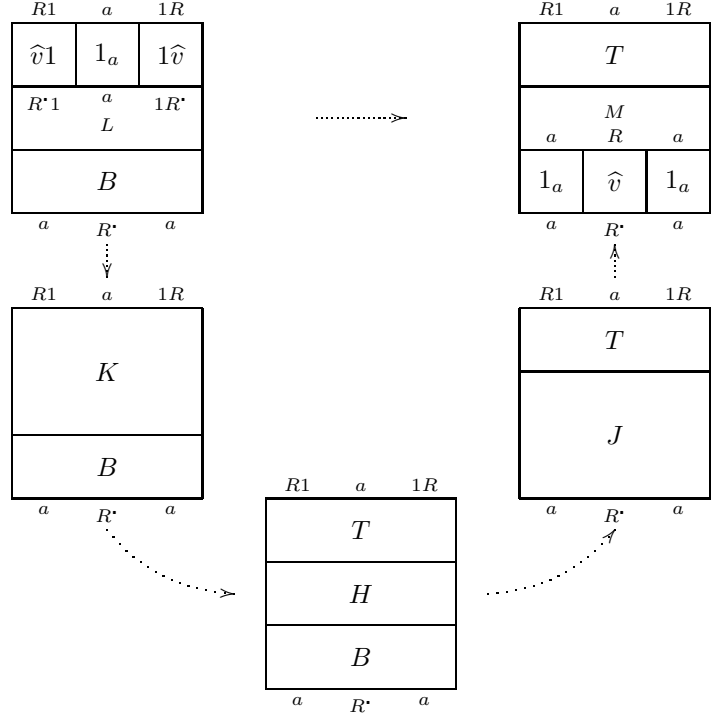
\[
\xy 0;/r.20pc/:
{\ar@{-}^{R1} (0,0)*{}; (10,0)*{} };
{\ar@{-}^{a} (10,0)*{}; (20,0)*{} };
{\ar@{-}^{1R} (20,0)*{}; (30,0)*{} };
{\ar@{-} (0,0)*{}; (0,-30)*{} };
{\ar@{-} (30,0)*{}; (30,-30)*{} };
{\ar@{-}_{R^{\dot}1} (0,-10)*{}; (10,-10)*{} };
{\ar@{-}_{a} (10,-10)*{}; (20,-10)*{} };
{\ar@{-}_{1R^{\dot}} (20,-10)*{}; (30,-10)*{} };
{\ar@{-} (10,0)*{}; (10,-10)*{} };
{\ar@{-} (20,0)*{}; (20,-10)*{} };
{\ar@{-} (0,-20)*{}; (30,-20)*{} };
{\ar@{-}_{a} (0,-30)*{}; (10,-30)*{} };
{\ar@{-}_{R^{\dot}} (10,-30)*{}; (20,-30)*{} };
{\ar@{-}_{a} (20,-30)*{}; (30,-30)*{} };
{\ar@{.>} (48,-15)*{}; (62,-15)*{} };
{\ar@{-}^{R1} (80,0)*{}; (90,0)*{} };
{\ar@{-}^{a} (90,0)*{}; (100,0)*{} };
{\ar@{-}^{1R} (100,0)*{}; (110,0)*{} };
{\ar@{-} (80,-10)*{}; (110,-10)*{} };
{\ar@{-}_{a} (80,-30)*{}; (90,-30)*{} };
{\ar@{-}_{R^{\dot}} (90,-30)*{}; (100,-30)*{} };
{\ar@{-}_{a} (100,-30)*{}; (110,-30)*{} };
{\ar@{-}^{a} (80,-20)*{}; (90,-20)*{} };
{\ar@{-}^{R} (90,-20)*{}; (100,-20)*{} };
{\ar@{-}^{a} (100,-20)*{}; (110,-20)*{} };
{\ar@{-} (80,0)*{}; (80,-30)*{} };
{\ar@{-} (110,0)*{}; (110,-30)*{} };
{\ar@{-} (90,-20)*{}; (90,-30)*{} };
{\ar@{-} (100,-20)*{}; (100,-30)*{} };
(5,-5)*{\widehat{v}1}; (15,-5)*{1_{a}}; (25,-5)*{1\widehat{v}};
(85,-25)*{1_{a}}; (95,-25)*{\widehat{v}}; (105,-25)*{1_{a}};
{\ar@{-}^{R1} (40,-75)*{}; (50,-75)*{} };
{\ar@{-}^{a} (50,-75)*{}; (60,-75)*{} };
{\ar@{-}^{1R} (60,-75)*{}; (70,-75)*{} };
{\ar@{-}_{a} (40,-105)*{}; (50,-105)*{} };
{\ar@{-}_{R^{\dot}} (50,-105)*{}; (60,-105)*{} };
{\ar@{-}_{a} (60,-105)*{}; (70,-105)*{} };
{\ar@{-} (40,-75)*{}; (40,-105)*{} };
{\ar@{-} (70,-75)*{}; (70,-105)*{} };
{\ar@{-} (40,-85)*{}; (70,-85)*{} };
{\ar@{-} (40,-95)*{}; (70,-95)*{} };
(55,-80)*{T}; (55,-100)*{B}; (55,-90)*{H};
(15,-25)*{B}; (95,-5)*{T};
{\ar@{-}^{R1} (0,-45)*{}; (10,-45)*{} };
{\ar@{-}^{a} (10,-45)*{}; (20,-45)*{} };
{\ar@{-}^{1R} (20,-45)*{}; (30,-45)*{} };
{\ar@{-}_{a} (0,-75)*{}; (10,-75)*{} };
{\ar@{-}_{R^{\dot}} (10,-75)*{}; (20,-75)*{} };
{\ar@{-}_{a} (20,-75)*{}; (30,-75)*{} };
{\ar@{-} (0,-45)*{}; (0,-75)*{} };
{\ar@{-} (30,-45)*{}; (30,-75)*{} };
{\ar@{-} (0,-65)*{}; (30,-65)*{} };
{\ar@{-}^{R1} (80,-45)*{}; (90,-45)*{} };
{\ar@{-}^{a} (90,-45)*{}; (100,-45)*{} };
{\ar@{-}^{1R} (100,-45)*{}; (110,-45)*{} };
{\ar@{-}_{a} (80,-75)*{}; (90,-75)*{} };
{\ar@{-}_{R^{\dot}} (90,-75)*{}; (100,-75)*{} };
{\ar@{-}_{a} (100,-75)*{}; (110,-75)*{} };
{\ar@{-} (80,-45)*{}; (80,-75)*{} };
{\ar@{-} (110,-45)*{}; (110,-75)*{} };
{\ar@{-} (80,-55)*{}; (110,-55)*{} };
{\ar@{.>} (15,-35)*{}; (15,-40)*{} };
{\ar@/_0.5pc/@{.>} (15,-80)*{}; (35,-90)*{} };
{\ar@{.>} (95,-40)*{}; (95,-35)*{} };
{\ar@/_0.5pc/@{.>} (75,-90)*{}; (95,-80)*{} };
(15,-70)*{B}; (95,-50)*{T}; (15,-55)*{K}; (95,-65)*{J};
(15,-16)*{\scriptstyle L};  (95,-14)*{\scriptstyle M}
\endxy
\]

The homotopy $K$ is given as
\[K(t,s)=\biggl(L_1(t,\dfrac{s}{2}+1), (1-\dfrac{s}{2})+(s-1)L_2(t,\dfrac{s}{2}+1),\hf \pm\dfrac{1}{4}\sin(\dfrac{\pi s}{2})\biggr),\]
where the plus or minus sign in the last coordinate depends on which little cube is being moved: for $a$ it is plus, and for $b$ and $c$ it is minus.

There is a homotopy $\Phi\colon [0,3]\times [0,2]\times [0,1] \to \sC_3(3)$ from the top two-thirds of the upper left function to $K$, with cases given for $0\leq s\leq 1$ and $1\leq s\leq 2$, respectively.

\[
\Phi(t,s,u)= \left\{ \begin{array}{l}
\scriptstyle \biggl(L_1(t,\dfrac{su}{2}+1), \, 1-s(1-\dfrac{u}{2})+(s(2-u)-1)L_2(t,\dfrac{su}{2}+1), \, \hf \pm \dfrac{1}{4}\sin(\pi s (1-\dfrac{u}{2}))\biggr)\\
\scriptstyle \biggl(L_1(t,s+(1-\dfrac{s}{2})u), \, (1-\dfrac{s}{2})u+((s-2)u+1)L_2(t,s+(1-\dfrac{s}{2})u), \, \hf \pm \dfrac{1}{4}\sin(\dfrac{\pi s u}{2})\biggr)
\end{array} \right.
\]
We leave it to the reader to check that this is indeed a homotopy as desired.  With similar methods we can construct the homotopy from $K$ to $HT$, and thus give the homotopy $\alpha_1$.  The analogous argument produces $\alpha_{2}$ as well, verifying the first syllepsis axiom.
\end{proof}

\section{Proof of coherence}

This is the heart of the paper, in which we give a proof of coherence for symmetric monoidal bicategories.  We begin by reducing the problem from computing in the free symmetric monoidal bicategory on a set of objects to computing in the ``positive braid part'' of that same free symmetric monoidal bicategory.  This allows us to use purely algebraic results about the positive braid monoid.  We then finish the proof using a rewriting argument.

\subsection{Positive braids}

This section will establish some basic results that will be necessary in reducing the coherence theorem to one which will be easily approachable by rewriting methods.  Recall that the free symmetric monoidal bicategory on one object has 1-cells which are freely generated by the structural 1-cells present in a symmetric monoidal bicategory:  associativity and unit constraints for the monoidal structure, and the braiding 1-cells $R, R^{\dot}$.  The focus of this section, then, is to remove the 1-cells $R^{\dot}$ from this structure.  Geometrically, this means replacing arbitrary braids with those only having over-crossings, in other words restricting to positive braids.  As positive braids, and the positive braid monoid, have been studied independently as purely algebraic objects, we will then be able to apply traditional combinatorial techniques.  A crucial first observation is the following.

\begin{lem}\label{RRadj}
There is an adjoint equivalence $R_{xy} \dashv_{eq} R_{yx}$ with unit given by $v_{xy}^{-1}$ and counit given by $v_{yx}$.
\end{lem}
\begin{proof}
This is an immediate consequence of the symmetry axiom.
\end{proof}

Using this lemma and the uniqueness of adjoints, we can construct a unique invertible 2-cell $t:R_{xy}^{\dot} \Rightarrow R_{yx}$ such that the equality of pasting diagrams below holds.
\[
\xy
{\ar@/^1pc/^{R_{xy}^{\dot}} (0,0)*+{xy}; (20,10)*+{yx} };
{\ar@/_1pc/_{R_{yx}} (0,0)*+{xy}; (20,10)*+{yx} };
{\ar@/^1pc/^{R_{xy}} (20,10)*+{yx}; (40,0)*+{xy} };
{\ar@/_2pc/_{1} (0,0)*+{xy}; (40,0)*+{xy} };
{\ar@/^1pc/^{R_{xy}^{\dot}} (60,0)*+{xy}; (80,10)*+{yx} };
{\ar@/^1pc/^{R_{xy}} (80,10)*+{yx}; (100,0)*+{xy} };
{\ar@/_2pc/_{1} (60,0)*+{xy}; (100,0)*+{xy} };
(10,6)*{\Downarrow t}; (24,-2)*{\Downarrow v_{xy}}; (50,0)*{=}; (80,0)*{\Downarrow \varepsilon_{xy}}
\endxy
\]

\begin{lem}\label{tisamod}
The components of $t$ given above define a modification.
\end{lem}
\begin{proof}
To show that $t$ is a modification, we must verify the equality of pastings below, where both of the isomorphisms in squares are naturality isomorphisms.
\[
\xy
{\ar@/^1.5pc/^{R_{xy}^{\dot}} (0,0)*+{yx}; (20,0)*+{xy} };
{\ar@/_1.5pc/_{R_{yx}} (0,0)*+{yx}; (20,0)*+{xy} };
{\ar_{gf} (0,0)*+{yx}; (0,-20)*+{y'x'} };
{\ar^{fg} (20,0)*+{xy}; (20,-20)*+{x'y'} };
{\ar@/_1.5pc/_{R_{y'x'}} (0,-20)*+{y'x'}; (20,-20)*+{x'y'} };
(10,-14)*{\cong}; (10,0)*{\Downarrow t};
{\ar@/^1.5pc/^{R_{xy}^{\dot}} (40,0)*+{yx}; (60,0)*+{xy} };
{\ar_{gf} (40,0)*+{yx}; (40,-20)*+{y'x'} };
{\ar^{fg} (60,0)*+{xy}; (60,-20)*+{x'y'} };
{\ar@/_1.5pc/_{R_{y'x'}} (40,-20)*+{y'x'}; (60,-20)*+{x'y'} };
{\ar@/^1.5pc/^{R_{y'x'}^{\dot}} (40,-20)*+{y'x'}; (60,-20)*+{x'y'} };
(50,-20)*{\Downarrow t}; (50,-6)*{\cong}; (30,-10)*{=}
\endxy
\]
Since all of the 1-cells involved are equivalences and the 2-cells involved are isomorphisms, the two pastings above are equal if and only if they are equal after we paste naturality squares
\[
\xy
{\ar^{R_{xy}} (0,0)*+{xy}; (20,0)*+{yx} };
{\ar^{gf} (20,0)*+{yx}; (20,-10)*+{y'x'} };
{\ar_{fg} (0,0)*+{xy}; (0,-10)*+{x'y'} };
{\ar_{R_{x'y'}} (0,-10)*+{x'y'}; (20,-10)*+{y'x'} };
(10,-5)*{\cong}
\endxy
\]
to the right of each, and then paste $v_{yx}$ along the bottom.  By the definition of $t$ and the fact that both $v$ and $\varepsilon$ are both modifications, it is simple to check that both resulting pastings are equal to the one shown below.
\[
\xy
{\ar^{R_{xy}^{\dot}} (0,0)*+{yx}; (20,0)*+{xy} };
{\ar^{R_{xy}} (20,0)*+{xy}; (40,0)*+{yx} };
{\ar@/_1.5pc/_{1} (0,0)*+{yx}; (40,0)*+{yx} };
{\ar_{gf} (0,0)*+{yx}; (0,-15)*+{y'x'} };
{\ar^{gf} (40,0)*+{yx}; (40,-15)*+{y'x'} };
{\ar_{1} (0,-15)*+{y'x'}; (40,-15)*+{y'x'} };
(20,-4)*{\Downarrow \varepsilon}; (20,-12)*{\cong}
\endxy
\]
\end{proof}

\noindent \textbf{Definition.}  Let $\mathcal{S}_{+}$ be the sub-bicategory of $\mathcal{S}$ with
\begin{itemize}
\item all of the objects of $\mathcal{S}$,
\item 1-cells being those containing no instances of a generating 1-cell of the form $R_{xy}^{\dot}$, and
\item 2-cells being all of those obtained by composing and tensoring the generating 2-cells of $\mathcal{S}$ with the property that their source and target 1-cells are in $\mathcal{S}_{+}$.
\end{itemize}

In particular, the 2-cells of $\mathcal{S}_{+}$ are those which can be formed from the 2-cell constraints arising from the monoidal structure, naturality 2-cells for $R$, the 2-cells $R_{-|--}, R_{--|-}$, and the 2-cells $v$.  These 2-cells do not have any instances of naturality 2-cells for $R^{\dot}$, and they do not have any instances of the unit and counit for the adjoint equivalence $R \dashv_{eq} R^{\dot}$.

The proof of the next proposition uses the theory of icons to prove that the inclusion of $\mathcal{S}_{+}$ into $\mathcal{S}$ is a symmetric monoidal biequivalence.  An icon is a kind of transformation that only exists between a pair of functors which agree on objects.  In order to show that a functor $F:X \rightarrow Y$ is a biequivalence, one usually constructs a pseudo-inverse $G:Y \rightarrow X$ and then proves that the composite functors $FG, GF$ are equivalent to $1_{Y}, 1_{X}$, respectively.  Giving such an equivalence requires constructing a pseudonatural transformation whose components are equivalences.  In our case, one of the composite functors is already equal to the identity, and the other is the identity on objects.  Thus instead of constructing a pseudonatural transformation, it suffices to give an icon whose 2-cell components are isomorphisms.  We refer the reader to \cite{lack, LP} for the general theory of icons.

\begin{prop}\label{positivebraidedstructure}
The bicategory $\mathcal{S}_{+}$ can be given the structure of a symmetric monoidal bicategory such that the inclusion $i: \mathcal{S}_{+} \hookrightarrow \mathcal{S}$ can then be equipped with the structure of a symmetric monoidal functor.  This functor $i$ is then a symmetric monoidal biequivalence.
\end{prop}
\begin{proof}
The monoidal structure for $\mathcal{S}_{+}$ is given by the same tensor product, unit, and constraints as that for $\mathcal{S}$.  The braided monoidal structure is also the same except that we define the adjoint equivalence $R_{xy} \dashv_{eq} R_{xy}^{\dot}$ to be given by $R_{xy} \dashv_{eq} R_{yx}$ using Lemma \ref{RRadj}; the axioms then follow from those in $\mathcal{S}$ as $R_{xy}^{\dot}$ does not appear as part of the standard collection of axioms.  We can also keep the same invertible modification $v$, and both the syllepsis axioms (using the version given by McCrudden in \cite{mccrudden}) and the symmetry axiom follow from those in $\mathcal{S}$.  The same arguments show that the inclusion $i$ can be equipped with the structure of a symmetric monoidal functor by defining $\chi$ to be the identity adjoint equivalence and $U$ to be the obvious composite of unit isomorphisms.

To show that $i$ is a symmetric monoidal biequivalence, we construct a pseudo-inverse $T:\mathcal{S} \rightarrow \mathcal{S}_{+}$ which will be a strict monoidal functor.  We define $T$ to be the identity on objects, and the identity on all of the generating 1-cells which are in both $\mathcal{S}$ and $\mathcal{S}_{+}$.  Define $T(R_{xy}^{\dot}) = R_{yx}$.  Similarly on 2-cells, we define $T$ to be the identity on all generating 2-cells which are in both $\mathcal{S}$ and $\mathcal{S}_{+}$.  Define $T(\eta_{xy}) = v_{xy}^{-1}$ and $T(\varepsilon_{xy}) = v_{yx}$.  Requiring that $T$ be strict monoidal then extends it to tensors and composites of all of these cells, thus defining the map on underlying 2-globular sets.

Now it is clear that $Ti=1$, so we need only show that $iT \simeq 1$.  Since these two functors agree on objects, we will show that there is an invertible icon $\alpha:iT \Rightarrow 1$.  On a generating 1-cell $f$ of $\mathcal{S}$, $\alpha_{f}$ is the identity except when $f = R_{xy}^{\dot}$, in which case we define $\alpha_{f}$ to be $t^{-1}$.  We extend this over composition, forcing the icon axioms to hold if the components of $\alpha$ we have given here are natural in $f$.  Note that every component is invertible, so in fact naturality is all we have to check.  We only must check naturality for generating 2-cells, as naturality for composites will follow, and we only have to check naturality if either the source or the target involves $R_{xy}^{\dot}$ as otherwise $iT$ is the identity.  We leave the full verification of these axioms to the reader, indicating in each case how naturality can be demonstrated.
\begin{itemize}
\item Naturality of $\alpha$ with respect to $\varepsilon$ is merely a rewritten version of the diagram defining $t$, so holds automatically.
\item Naturality of $\alpha$ with respect to $\eta$ follows from the triangle identities for $\eta$ and $\varepsilon$ and the definition of $t$ starting from the 2-cell $1_{R_{xy}} * \eta_{xy}$.
\item Naturality of $\alpha$ with respect to any of the 2-cells $a,l,r$ just follows from the naturality of those cells.
\item Naturality of $\alpha$ with respect to the naturality 2-cells for $R$ is immediate, and with respect to the naturality 2-cells for $R^{\dot}$ is that $t$ is a modification.
\end{itemize}
\end{proof}

This proposition completes the task of replacing $\mathcal{S}$ with a symmetric monoidal bicategory which, geometrically, has 1-cells which correspond to positive braids.  This reduces proving Theorem \ref{coherencesmb} to proving the analogous statement for $\mathcal{S}_{+}$.  We focus now on that task.

\noindent \textbf{Definition.}  Let $f$ be a 1-cell in $\mathcal{S}$.  The \textit{underlying permutation} of $f$ is the element $\pi(f) \in \Sigma_{n}$ where $\pi:\mathcal{S} \rightarrow \Sigma$ is the symmetric monoidal functor induced by the universal property.  If $f$ is a 1-cell in $\mathcal{S}_{+}$, then its underlying permutation is $\pi i(f)$ where $i:\mathcal{S}_{+} \hookrightarrow \mathcal{S}$.

\begin{rem}
It is important to note that if $\alpha:f \Rightarrow g$ is a 2-cell in $\mathcal{S}$, then $\pi(f) = \pi(g)$.  This is easy to check, as none of the generating 2-cells in $\mathcal{S}$ change the underlying permutation.  Later, we will prove a converse to this:  if $\pi(f) = \pi(g)$, then there exists a 2-cell $\alpha:f \Rightarrow g$ in $\mathcal{S}$.  The fact that this 2-cell will also be unique is then one expression of coherence for symmetric monoidal bicategories.
\end{rem}

\noindent \textbf{Definition.}  Let $f$ be a 1-cell in $\mathcal{S}$.  Since $\mathcal{S}$ has the same underlying graph (by which we mean 0- and 1-cells, together with source and target information) as the free braided monoidal bicategory on $*$, here denoted $\mathcal{B}$ (see \cite{gurski-braid}), we consider the 1-cell $f$ now in $\mathcal{B}$.  The universal property of $\mathcal{B}$ induces a strict braided monoidal functor $\rho:\mathcal{B} \rightarrow \mb{Br}$ where $\mb{Br}$ is the collection of braid groups, viewed as a locally discrete braided monoidal bicategory.  The \textit{underlying braid} of $f$ is then $\rho(f)$, viewed as an element of a braid group.  If $f$ is now a 1-cell in $\mathcal{S}_{+}$, the underlying braid of $f$ is then the underlying braid of $i(f)$.

\begin{rem}
Note that for any 1-cell in $\mathcal{S}_{+}$, the underlying braid is a positive braid, i.e., can be written as a word in the generators $\sigma_{i}$ without the use of any inverses.  Conversely, given a positive braid $\alpha$ we can construct a 1-cell $f$ in $\mathcal{S}_{+}$ such that the underlying braid of $f$ is $\alpha$.  The full coherence theorem for symmetric monoidal bicategories will give a correspondence between symmetries and 1-cells in $\mathcal{S}_{+}$ (taken up to isomorphism).  The coherence theorem for braided monoidal bicategories already implies a limited version of that result, namely that given any two 1-cells $f,f'$ in $\mathcal{S}_{+}$ with the same underlying braid, there is a unique isomorphism $f \cong f'$ using only the braided structure of $\mathcal{S}_{+}$, i.e., a unique invertible 2-cell constructed without any instances of $v, v^{-1}$.  We will implicitly be using this correspondence at the braided level in our rewriting strategy for coherence.
\end{rem}

\noindent \textbf{Definition.}  Let $\rho$ be a positive braid on $n$ strands.  Then the \textit{starting set} $S( \rho) \subseteq \{ 1, \ldots, n-1 \}$ is the set
\[
\{ i: \rho = \sigma_{i} \tau, \tau \textrm{ a positive braid} \}.
\]
The \textit{finishing set} $F( \rho) \subseteq \{ 1, \ldots, n-1 \}$ is the set
\[
\{ i: \rho = \tau \sigma_{i} , \tau \textrm{ a positive braid} \}.
\]

\noindent \textbf{Definition.}  For a positive braid $\rho$, a factorization $\rho = \tau \omega$ into a product of two positive braids is \textit{left-weighted} if $S(\omega) \subseteq F(\tau)$.

\begin{rem}
If $\rho = \tau \omega$ is a left-weighted factorization, then for any $i \in S(\omega)$, we can write
\[
\rho = \tau' \sigma_{i}^{2} \omega'.
\]
This is the property that we will need later in order to prove that every 1-cell is isomorphic to a minimal one (see the following definition).
\end{rem}

\noindent \textbf{Definition.}  A 1-cell $f$ in $\mathcal{S}_{+}$ is \textit{minimal} if its underlying braid has the property that no two strands cross twice.  We will also refer to a positive braid as minimal if it has this property.

The following lemma can be found in \cite{EM2}.

\begin{lem}
For any positive braid $\rho$, there is a left-weighted factorization $\rho = \tau \omega$ in which $\tau$ is minimal.
\end{lem}

\begin{prop}\label{minimal}
Every 1-cell $f$ in $\mathcal{S}_{+}$ is isomorphic to a minimal 1-cell $f'$.
\end{prop}
\begin{proof}
Let $\rho(f)$ be the underlying positive braid of $f$.  Using the left-weighted factorization of $\rho(f)$ into a product of two positive braids, we can write
\[
\rho(f) = b_{1} \sigma_{i}^{2} b_{2}
\]
for some positive braids $b_{1}, b_{2}$ and some value of $i$.  Using only 2-cells from the braided structure of $\mathcal{S}_{+}$, $f$ is isomorphic to some 1-cell $g$ whose underlying braid is precisely $b_{1} \sigma_{i}^{2} b_{2}$, and then we can use an instance of $v$ to remove the $\sigma_{i}^{2}$.  This procedure has reduced the total number of crossings, and can be continued until no two strands cross twice.
\end{proof}

\subsection{Rewriting techniques}

To prove our coherence theorem, we will use rewriting techniques.    We will define a notion of abstract reduction which changes one positive braid into another, and use these to give a kind of normal form for 2-cells in $\mathcal{S}_{+}$.  Throughout this section, we will always use $\alpha$ to denote a positive braid.

\noindent \textbf{Definition.}  The \textit{basic reductions} are given by
\begin{itemize}
\item the reductions of type YB are
\[
\begin{array}{rcl}
\sigma_{i} \sigma_{i+1} \sigma_{i} & \leadsto & \sigma_{i+1} \sigma_{i} \sigma_{i+1}, \\
\sigma_{i+1} \sigma_{i} \sigma_{i+1} & \leadsto & \sigma_{i} \sigma_{i+1} \sigma_{i},
\end{array}
\]
\item the reductions of type C are $\sigma_{i} \sigma_{j} \leadsto \sigma_{j} \sigma_{i}$ for $|i-j| >1$, and
\item the reductions of type V are $\sigma_{i} \sigma_{i} \leadsto 1$.
\end{itemize}
We also have \textit{composite basic reductions}; those of type YB are given by
\[
\begin{array}{rcl}
\sigma_{i} \sigma_{i+1} \sigma_{i}^{n} & \leadsto & \sigma_{i+1}^{n} \sigma_{i} \sigma_{i+1}, \\
\sigma_{i}^{n} \sigma_{i+1} \sigma_{i} & \leadsto & \sigma_{i+1} \sigma_{i} \sigma_{i+1}^{n}, \\
\sigma_{i+1} \sigma_{i} \sigma_{i+1}^{n} & \leadsto & \sigma_{i}^{n} \sigma_{i+1} \sigma_{i},\\
\sigma_{i+1}^{n} \sigma_{i} \sigma_{i+1} & \leadsto & \sigma_{i} \sigma_{i+1} \sigma_{i}^{n},
\end{array}
\]
and those of type C are given by
\[
\begin{array}{rcl}
\sigma_{i}^{n} \sigma_{j} & \leadsto & \sigma_{j} \sigma_{i}^{n}, \\
\sigma_{i} \sigma_{j}^{n} & \leadsto & \sigma_{j}^{n} \sigma_{i}.
\end{array}
\]
In addition, if $\alpha, \alpha'$ are positive braids and we have a basic reduction $\beta \leadsto \widehat{\beta}$, then there is a basic reduction
\[
\alpha \beta \alpha' \leadsto \alpha \widehat{\beta} \alpha'.
\]

\begin{rem}
The basic reductions of type YB are denoted as such because of their relation to the Yang-Baxter equation.  The basic reductions of type C are commutativity relations which are also present at the braided level.  The reductions of type V are part of the symmetric structure, and correspond to the relation $\sigma_{i}^{2}=1$ in the standard presentation of the symmetric groups.
\end{rem}

\noindent \textbf{Definition.}  A \textit{reduction} is a chain of basic reductions
\[
\alpha_{1} \leadsto \alpha_{2} \leadsto \cdots \leadsto \alpha_{n};
\]
we call $\alpha_{1}$ the source and $\alpha_{n}$ the target.

\noindent \textbf{Definition.}  The \textit{realization} of the basic reductions is given below; note that these are 2-cells in $\mathcal{S}_{+}$.
\begin{itemize}
\item The realization of the reduction $\sigma_{i} \sigma_{i+1} \sigma_{i}  \leadsto  \sigma_{i+1} \sigma_{i} \sigma_{i+1}$ of type YB is the isomorphism (tensored with identities on each side, if necessary)
\[
\begin{array}{rcl}
a \circ R1 \circ a^{\dot} \circ 1R \circ a \circ R1 & \stackrel{1*R_{-|--}}{\Longrightarrow} & a \circ R1 \circ a^{\dot} \circ a \circ R \circ a \\
& \cong & a \circ R1 \circ R \circ a \\
& \stackrel{\textrm{naturality}}{\Longrightarrow} & a \circ  R \circ 1R \circ a \\
& \cong & a \circ  R \circ a \circ a^{\dot} \circ 1R \circ a \\
& \stackrel{R_{-|--}^{-1}*1}{\Longrightarrow} & 1R \circ a \circ R1 \circ a^{\dot} \circ 1R \circ a.
\end{array}
\]
\item The realization of the reduction $\sigma_{i+1} \sigma_{i} \sigma_{i+1}  \leadsto  \sigma_{i} \sigma_{i+1} \sigma_{i}$ of type YB is the isomorphism
\[
\begin{array}{rcl}
a^{\dot} \circ 1R \circ a \circ R1 \circ a^{\dot} \circ 1R & \stackrel{1*R_{--|-}}{\Longrightarrow} & a^{\dot} \circ 1R \circ a \circ a^{\dot} \circ R \circ a^{\dot} \\
& \cong & a^{\dot} \circ 1R \circ R \circ a^{\dot} \\
& \stackrel{\textrm{naturality}}{\Longrightarrow} & a^{\dot} \circ R \circ R1 \circ a^{\dot} \\
& \cong & a^{\dot} \circ R \circ a^{\dot} \circ a \circ R1 \circ a^{\dot} \\
& \stackrel{R_{--|-}^{-1}*1}{\Longrightarrow}& R1 \circ a^{\dot} \circ 1R \circ a \circ R1 \circ a^{\dot}.
\end{array}
\]
The realization of the composite basic reductions of type YB are defined in an analogous manner, the only change being that the naturality 2-cells for $R$ are with respect to a composite 1-cell instead of a single instance of $R$.
\item The realization of C is the appropriate instance of the composite functoriality isomorphism for $\otimes$ as a functor
\[
\begin{array}{rcl}
(1 \otimes f) \circ (g \otimes 1) & \cong & (1 \circ g) \otimes (f \circ 1) \\
& \cong & g \otimes f \\
& \cong & (g \circ 1) \otimes (1 \circ f) \\
& \cong & (g \otimes 1) \circ (1 \otimes f)
\end{array}
\]
where $f,g$ are morphisms of the form $1^{\otimes n} \otimes R \otimes 1^{\otimes m}$.
\item The realization of V is $v$.
\end{itemize}
In addition, for any basic reduction of the form $\alpha \beta \alpha' \leadsto \alpha \widehat{\beta} \alpha'$, the realization is defined to be the realization of $\beta \leadsto \widehat{\beta}$ whiskered by the identities on $\alpha, \alpha'$.   The realization of a reduction is given as the 2-cell in $\mathcal{S}_{+}$ obtained from by composing the realizations of each individual basic reduction.

Strictly speaking, there is some ambiguity in the definition above that arises from not picking source and target 0-cells, but we assume once and for all that any necessary choices have been made.  There is then the further issue that such choices may not allow composition of the 2-cells making up the realization of a reduction, but this problem is solved by invoking the coherence theorem for monoidal bicategories and inserting unique 2-cell isomorphisms where necessary.  Taking this into account proves the following lemma which will operate in the background throughout the rest of the argument.

\begin{lem}
\begin{enumerate}
\item Any 2-cell $\Gamma:f \Rightarrow g$ in $\mathcal{S}_{+}$ not involving $v^{-1}$ gives rise to a reduction $\textrm{red}(\Gamma):f \leadsto g$ by assigning to each generating 2-cell in $\Gamma$ the obvious reduction which realizes it, up to constraints arising solely from the monoidal structure.
\item Conversely, any reduction $r:\alpha \leadsto \beta$ between positive braids gives rise to a 2-cell $\textrm{real}(r):\overline{\alpha} \Rightarrow \overline{\beta}$, where $\overline{\alpha}, \overline{\beta}$ are 1-cells with underlying positive braids $\alpha, \beta$, respectively.
\item The equation $\textrm{real} \circ \textrm{red}(\Gamma) = \Gamma$ can be made to hold (once again, for $\Gamma$ not involving instances of $v^{-1}$) if we choose the 1-cells $\overline{f}, \overline{g}$ to be $f, g$, respectively, as we take the same generating 2-cells as appear in $\Gamma$ in the same order and use coherence for monoidal bicategories to take care of any associativity or units.
\item We can make these choices such that $\textrm{red} \circ \textrm{real}(r)=r$ for any reduction $r$.
\end{enumerate}
\end{lem}

\noindent \textbf{Definition.}  Two reductions are equal if they have the same realizations in $\mathcal{S}_{+}$.

\noindent \textbf{Definition.}  A reduction $\alpha_{1} \leadsto \cdots \leadsto \alpha_{n}$ is \textit{complete} if $\alpha_{n}$ is a minimal positive braid.

\begin{lem}\label{l=0}
If $\alpha, \beta$ are both minimal positive braids with the same underlying permutation, then there is a unique reduction $\alpha \leadsto \cdots \leadsto \beta$.
\end{lem}
\begin{proof}
By \cite{EM2} minimal positive braids with the same underlying permutation are necessarily equal elements of the braid group, hence they give parallel 1-cells in the free braided monoidal bicategory on one object which are uniquely isomorphic by \cite{gurski-braid}.  By minimality, any reduction of $\alpha$ only uses basic reductions of type YB or C, hence realizes to a 2-cell in the free braided monoidal bicategory.  Thus the coherence theorem in the braided case gives existence and uniqueness.
\end{proof}

Our key result is the following.  We devote the next section to its proof.

\begin{thm}\label{completered}
For a positive braid $\alpha$, any two complete reductions with the same target are equal.
\end{thm}

\subsection{Marked braids}

In order to prove Theorem \ref{completered}, we must introduce some notation and terminology.

\noindent \textbf{Definition.}  Let $A$ be a set.  A \textit{marking function} with labels in $A$ is a function $m: A \coprod A \rightarrow \mathbb{N}$, or equivalently a function $A \rightarrow \mathbb{N}^{2}$.

\begin{rem}
In the proofs below, we are really only interested in the case when $A$ has one or two elements.  This set is just used for bookkeeping purposes, and has no intrinsic properties.
\end{rem}

Let $\alpha$ be a 1-cell in $\mathcal{S}_{+}$.  By construction, this 1-cell has a length $n$ which is the number of crossings that appear in the geometric braid representing $\alpha$, or equally the number of generating 1-cells of the form $R_{x,y}$ that appear in $\alpha$ so long as we consider 1-cells in which $R$'s always braid a single object past another single object (no $R$ of the form $R_{x,y\otimes z}$).  We can always assume that $\alpha$ is a 1-cell as just described (and we shall for the remainder of this discussion unless otherwise noted).  A marking function $m: A \rightarrow \mathbb{N}^{2}$ whose image consists of pairs with each coordinate in $\{1, \ldots, n \}$ can be used to mark the braid $\alpha$ by marking the $i$th crossing and $j$th crossing in the geometric braid representing $\alpha$ with $a$ if $m(a)=(i,j)$.  Conversely, given a geometric braid with pairs of crossings marked with the elements of $A$, we can construct a marking function by sending the element $a$ to the pair of
natural numbers $(i,j)$ where the $i$th and $j$th crossings are labeled by $a$.

Finally, we can use reductions of type YB or type C to transfer markings from one braid to another.  For the definition below, we use the notation $\sigma_{i}^{a}$ to indicate that the crossing $\sigma_{i}$ is marked by the element $a$.

\noindent \textbf{Definition.}  Let $\alpha$ be a 1-cell in $\mathcal{S}_{+}$ equipped with a marking function $m:A \coprod A \rightarrow \mathbb{N}$.  Then the \textit{transferred marking} under a basic reduction $\alpha \leadsto \beta$ of type YB or C is given below; we have not listed the composite basic reductions, as they are derived from these in the obvious fashion.
\[
\begin{array}{rcl}
\sigma_{i}^{a} \sigma_{i+1}^{b} \sigma_{i}^{c} & \leadsto & \sigma_{i+1}^{c} \sigma_{i}^{b} \sigma_{i+1}^{a} \\
\sigma_{i+1}^{a} \sigma_{i}^{b} \sigma_{i+1}^{c} & \leadsto & \sigma_{i}^{c} \sigma_{i+1}^{b} \sigma_{i}^{a} \\
\sigma_{i}^{a} \sigma_{j}^{b} & \leadsto & \sigma_{j}^{b} \sigma_{i}^{a}
\end{array}
\]
Note that we allow any of the labels to be omitted in which case both crossings with that label are unmarked.

Consider the following collection of reductions, in which each reduction $f_{j}$ is either of type C or type YB and the reductions $v,v'$ are of type V.
\[
\xy
{\ar@{~>}_{v} (0,0)*+{\alpha_{1}}; (0,-15)*+{\beta} };
{\ar@{~>}^{f_{1}} (0,0)*+{\alpha_{1}}; (25,0)*+{\alpha_{2}} };
{\ar@{~>}^{f_{2}} (25,0)*+{\alpha_{2}}; (50,0)*+{\cdots} };
{\ar@{~>}^{f_{k}} (50,0)*+{\cdots}; (75,0)*+{\alpha_{k+1}}};
{\ar@{~>}^{v'} (75,0)*+{\alpha_{k+1}}; (100,0)*+{\gamma} }
\endxy
\]
We can mark $\alpha_{1}$ using the set $\{x \}$ so that the two crossings which are removed by $v$ are marked.  Similarly, we can mark $\alpha_{k+1}$ using the set $\{ y \}$ so that the two crossings removed by $v'$ are marked.  Now for each reduction $f_{j}$, there is the inverse reduction $f_{j}^{-}:\alpha_{j+1} \leadsto \alpha_{j}$ which is of type YB or C if $f_{j}$ is.  Using the reductions
\[
\xy
{\ar@{~>}^{f_{k}^{-}} (0,0)*+{\alpha_{k+1}}; (25,0)*+{\cdots} };
{\ar@{~>}^{f_{1}^{-}} (25,0)*+{\cdots}; (50,0)*+{\alpha_{1}} }
\endxy
\]
we can transfer the marking of $\alpha_{k+1}$ using $\{ y \}$ to a marking of $\alpha_{1}$ using $\{ y \}$.  Thus the crossings marked with $y$'s are those which will be removed by $v'$, after the application of the $f_{j}$.  Taking the disjoint union of these two markings, we get a composite marking with labels in $\{x,y\}$ on $\alpha_{1}$.  We call this the \textit{canonical marking} $M$ of $\alpha_{1}$, and we refer to this collection of reductions and its canonical marking as the \textit{generic situation}.

We now prove three lemmas which constitute the technical bulk of the proof of Theorem \ref{completered}.  The strategy for each lemma is relatively simple.  In each case, we are given reductions with the same source as in the generic situation.  The goal is to find reductions $\beta \leadsto \sigma, \gamma \leadsto \sigma$ such that the two reductions $\alpha_{1} \leadsto \sigma$ are equal.  Since $\alpha_{1}$ is marked using the set $\{ x, y \}$ and, by construction, the two crossings marked with $x$ (resp., $y$) are distinct, there are three possible cases:  both of the crossings marked with $x$ are also the crossings marked with $y$, exactly one crossing marked $x$ is also marked $y$, or no crossing is marked with both $x$ and $y$.  Each lemma examines one of these three cases.

In each lemma, the following geometric idea is translated into algebra.  If $\alpha = \alpha_{1} \sigma_{i}^{2} \alpha_{2}$ is a positive braid, consider the geometric braid representing it and mark both of the crossings given by the $\sigma_{i}$'s.  Now assume we have another positive braid $\beta = \beta_{1} \sigma_{j}^{2} \beta_{2}$, and that $\alpha = \beta$; moreover, assume that the two crossings marked in $\alpha$ are the same crossings given by the $\sigma_{j}$'s in $\beta$.  Since $\alpha = \beta$, we can find a homotopy between the two geometric braids representing these positive braids, and we can choose this homotopy so that, given $\varepsilon >0$, the distance between the marked crossing (by distance, we mean the difference in height if the braids are seen as embedded in $\mathbb{R}^{3}$ beginning in the plane $z=1$ and ending in the plane $z=0$) is, at each time, less than $\varepsilon$.  In other words, the homotopy can be chosen to move the two marked crossings together as a pair.

\begin{lem}
Assume we are in the generic situation.  If $M$ has the property that the pair of crossings marked $x$ are the same as the pair of crossing marked $y$, then there are reductions $\beta \leadsto \sigma, \gamma \leadsto \sigma$ such that the two reductions $\alpha_{1} \leadsto \sigma$ are equal.
\end{lem}
\begin{proof}
We can change the sequence of reductions $f_{j}$ into a new sequence $g_{j}$ by choosing one of the two crossing marked with an $x$ and then using the following algorithm; once again, we assume that each $f_{i}$ is a basic reduction which is not a composite basic reduction.  If $f_{1}$ is a reduction that leaves the two marked crossings alone, then $g_{1}=f_{1}$.  If it moves the chosen marked crossing, then replace $f_{1}$ with a composite basic reduction of the same type that moves both marked crossings; note that the two marked crossings are adjacent, since they are in the source of $v$, and will continue to be adjacent after applying this composite basic reduction.  If $f_{1}$ moves the unchosen marked crossing, omit it.  Then for each $f_{j}$, we either replace it with a reduction moving the pair of marked crossings, keep the same reduction if it does not alter either marked crossing (by same, we mean a reduction of the same type moving the same crossings geometrically), or omit it if it moves the
unchosen marked crossing.  By coherence for braided monoidal bicategories, the sequence of reductions given by the $g_{j}$ is the same as that given by the $f_{j}$.  We also have that there are reductions $v_{2}:\alpha_{2} \leadsto \beta_{2}, h_{1}:\beta \leadsto \beta_{2}$ such that the two composite reductions $\alpha_{1} \leadsto \beta_{2}$ are equal using the naturality of $g_{1}$; continuing to use naturality finishes the result.
\end{proof}

\noindent \textbf{Definition.}  Let $m: A \coprod A \rightarrow \mathbb{N}$ be a marking function.  Then the \textit{distance function} associated with $m$ is the function $d:A \rightarrow \mathbb{N}$ given by
\[
d(a) = |mi_{1}(a) - mi_{2}(a)|,
\]
where $i_{1}, i_{2}$ are the two inclusions of $A$ into $A \coprod A$.

\begin{rem}
The proof above can be interpreted as saying that the reductions $f_{j}$ can be replaced by reductions $g_{j}$ such that the target of $g_{j}$ inherits a marking with the property that its distance function is constant at the value $1$.  This will be useful for the next two lemmas.
\end{rem}

\begin{lem}
Assume we are in the generic situation.  If $M$ has the property that one of the crossings marked $x$ is the same as one of the crossing marked $y$, then there are reductions $\beta \leadsto \sigma, \gamma \leadsto \sigma$ such that the two reductions $\alpha_{1} \leadsto \sigma$ are equal.
\end{lem}
\begin{proof}
Proceed as in the proof of the previous lemma, with the change that the crossing marked with both $x$ and $y$ is the chosen one.  Alter, retain, or omit reductions using the algorithm based on $x$ markings, and continue until the distance function for the marking has both $d(x)=1$ and $d(y)=1$; finish using the same algorithm, but now using the $y$ markings, once again with the chosen crossing being the one marked with both $x$ and $y$.  This produces the reductions below, where the $g_{i}$ preserve the property that $d(x)=1$ and the $h_{j}$ preserve the property that $d(y)=1$.
\[
\xy
{\ar@{~>}_{v} (0,0)*+{\alpha_{1}}; (0,-15)*+{\beta} };
{\ar@{~>}^{g_{1}} (0,0)*+{\alpha_{1}}; (20,0)*+{\cdots} };
{\ar@{~>}^{g_{m}} (20,0)*+{\cdots}; (40,0)*+{\alpha_{m+1}} };
{\ar@{~>}^{h_{1}} (40,0)*+{\alpha_{m+1}}; (60,0)*+{\cdots} };
{\ar@{~>}^{h_{n}} (60,0)*+{\cdots}; (80,0)*+{\alpha_{m+n+1}} };
{\ar@{~>}^{v'} (80,0)*+{\alpha_{m+n+1}}; (100,0)*+{\gamma} }
\endxy
\]
As in the previous lemma, we can use naturality to produce reductions $G:\beta \leadsto \beta_{2}, v_{x}:\alpha_{m+1} \leadsto \beta_{2}$ such that the two composite reductions $\alpha_{1} \leadsto \beta_{2}$ are equal; the reduction $v_{x}$ can be taken to be a single reduction of type V which removes the crossings marked $x$.  But by the symmetry axiom, this is equal to a single reduction $v_{y}$ of type V which removes the two crossings marked $y$ since $\alpha_{m+1}$, locally around the marked crossings, is given by the picture below (perhaps with $x$'s and $y$'s switched).
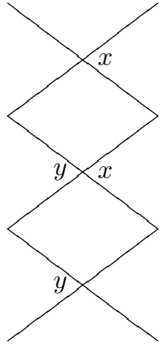
\[
\xy
(0,0)*{}; (20,-15)*{} **\dir{-};
(20,0)*{}; (0,-15)*{} **\dir{-};
(0,-15)*{}; (20,-30)*{} **\dir{-};
(20,-15)*{}; (0,-30)*{} **\dir{-};
(0,-30)*{}; (20,-45)*{} **\dir{-};
(20,-30)*{}; (0,-45)*{} **\dir{-};
(13,-7.5)*{x}; (13,-22.5)*{x}; (7,-22.5)*{y}; (7,-37.5)*{y}
\endxy
\]
Then the corresponding naturality argument using $v_{y}$ allows us to complete the proof.
\end{proof}

\begin{lem}
Assume we are in the generic situation.  If $M$ has the property that the four marked crossings are all distinct, then there are reductions $\beta \leadsto \sigma, \gamma \leadsto \sigma$ such that the two reductions $\alpha_{1} \leadsto \sigma$ are equal.
\end{lem}
\begin{proof}
We use the analogous algorithm by choosing one of the crossings marked with an $x$, obtaining a sequence of reductions $g_{i}$.  The target of the final reduction $g_{k}$ is some positive braid $\alpha_{k+1}'$ which has the property that the induced marking has distance function constant at $1$.  As a geometric braid (but \textit{not} as a word in the braid generators), $\alpha_{k+1}$ and $\alpha_{k+1}'$ are equal, so we can find some sequence of reductions $h_{i}$ of the form YB or C starting at $\alpha_{k+1}'$ and ending at $\alpha_{k+1}$.  We can additionally demand that the $h_{i}$ preserve the property that $d(y)=1$ by following the same algorithm applied to $y$ instead of $x$.  We then have that the reduction consisting of the $f_{i}$ equals that of the $g_{i}$ and $h_{i}$ by coherence for braided monoidal bicategories.

Then naturality of the 2-cells corresponding to the $g_{i}$ with respect to the 2-cell corresponding to $v$ shows that there are reductions $v_{k+1}:\alpha_{k+1}' \leadsto \beta_{k+1}, G:\beta \leadsto \beta_{k+1}$ such that the two composite reductions $\alpha_{1} \leadsto \beta_{k+1}$ are equal.  We can similarly use naturality for $v'$ with respect to the $h_{i}$ to produce the reductions shown below,  in which every enclosed region consists of a pair of equal reductions.
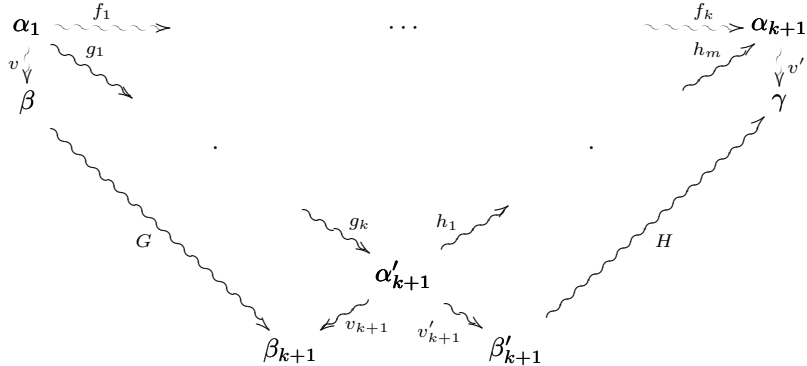
\[
\xy
{\ar@{~>}^{f_{1}} (0,0)*+{\alpha_{1}}; (20,0)*+{} };
{\ar@{~>}^{f_{k}} (80,0)*+{}; (100,0)*+{\alpha_{k+1}} };
(50,0)*{\cdots};
{\ar@{~>}^{g_{1}} (0,0)*+{\alpha_{1}}; (15,-10)*+{} };
{\ar@{~>}^{g_{k}} (35,-23)*+{}; (50,-33)*+{\alpha_{k+1}'} };
{\ar@{~>}^{h_{1}} (50,-33)*+{\alpha_{k+1}'}; (65,-23)*+{} };
{\ar@{~>}^{h_{m}} (85,-10)*+{}; (100,0)*+{\alpha_{k+1}} };
{\ar@{~>}_{v} (0,0)*+{\alpha_{1}}; (0,-10)*+{\beta} };
{\ar@{~>}_{G} (0,-10)*+{\beta}; (35,-43)*+{\beta_{k+1}} };
{\ar@{~>}^{v_{k+1}} (50,-33)*+{\alpha_{k+1}'}; (35,-43)*+{\beta_{k+1}} };
{\ar@{~>}_{v_{k+1}'} (50,-33)*+{\alpha_{k+1}'}; (65,-43)*+{\beta_{k+1}'} };
{\ar@{~>}^{v'} (100,0)*+{\alpha_{k+1}}; (100,-10)*+{\gamma} };
{\ar@{~>}_{H} (65,-43)*+{\beta_{k+1}'}; (100,-10)*+{\gamma} };
(25,-16)*{.}; (75,-16)*{.}
\endxy
\]
Note that $H$ can only consist of basic reductions of type YB or C since the reduction $\alpha_{k+1}' \leadsto \gamma$ involving the $h_{i}$ only has a single instance of a reduction of type V, hence the same must hold for the reduction involving $H$ and that is $v_{k+1}'$.  Thus to complete the required equality of reductions, we need only fill in the square at the bottom of this diagram and then use the ``inverse'' reduction of $H$.  But since $v_{k+1}$ and $v_{k+1}'$ remove two different sets of crossings by assumption, the square can be filled in to the square
\[
\xy
{\ar@{~>}^{v_{k+1}'} (0,0)*+{\alpha_{k+1}'}; (40,0)*+{\beta_{k+1}'} };
{\ar@{~>}_{v_{k+1}} (0,0)*+{\alpha_{k+1}'}; (0,-15)*+{\beta_{k+1}} };
{\ar@{~>} (40,0)*+{\beta_{k+1}'}; (40,-15)*+{\sigma} };
{\ar@{~>} (0,-15)*+{\beta_{k+1}}; (40,-15)*+{\sigma} };
\endxy
\]
using two different reductions of type V by the functoriality of both the tensor product and horizontal composition in $\mathcal{S}_{+}$.
\end{proof}

\noindent \textbf{Definition.}  The \textit{length} of a reduction is the number of basic reductions of type V in it.  For any positive braid $\alpha$, let $\ell(\alpha)$ denote the maximum length of a complete reduction.

We are finally in a position to prove Theorem \ref{completered}.  With our three technical lemmas in place, this is largely routine.

\begin{proof}[Proof of Theorem \ref{completered}]
We will induct over $\ell({\alpha})$.  Lemma \ref{l=0} is the case when $\ell(\alpha)=0$.  Assume the result is true for $\ell(\alpha) \leq n$, and let $\alpha_{1}$ be a positive braid of length $n+1$.  Changing the source 1-cell of $\alpha_{1}$ using only reductions of type YB or C if necessary, assume that we have two complete reductions  $\alpha_{1} \leadsto \omega$ which we will call $r_{1}$ and $r_{2}$,  one of which begins with a reduction $v$ of type V and the other with a sequence of reductions $f_{j}$ followed by $v'$ as in the generic situation.  By necessity, the canonical marking shows that the beginning of this pair of complete reductions falls into one of the cases by the three lemmas above, so we can complete this to a pair of equal reductions as shown below.
\[
\xy
{\ar@{~>}_{v} (0,0)*+{\alpha_{1}}; (0,-15)*+{\beta} };
{\ar@{~>}^{f_{1}} (0,0)*+{\alpha_{1}}; (20,0)*+{\alpha_{2}} };
{\ar@{~>}^{f_{2}} (20,0)*+{\alpha_{2}}; (30,0)*+{\cdots} };
{\ar@{~>}^{f_{k}} (30,0)*+{\cdots}; (50,0)*+{\alpha_{k+1}}};
{\ar@{~>}^{v'} (50,0)*+{\alpha_{k+1}}; (70,0)*+{\gamma} };
{\ar@{~>} (0,-15)*+{\beta}; (70,-15)*+{\sigma} };
{\ar@{~>} (70,0)*+{\gamma}; (70,-15)*+{\sigma} };
\endxy
\]

Now we have the remainder of the first complete reduction $\gamma \leadsto \omega$, and we can also choose any complete reduction $\sigma \leadsto \omega$.  Note that any complete reduction $\sigma \leadsto \delta$ has the property that $\omega = \delta$ as elements of the braid group, thus we can produce a complete reduction $\sigma \leadsto \omega$ by Lemma \ref{l=0}.  Now $\ell(\gamma) < \ell(\alpha_{1})$, so by induction the complete reduction $\gamma \leadsto \omega$ given by the remainder of $r_{1}$ is equal to the complete reduction $\gamma \leadsto \sigma \leadsto \omega$.  The same argument shows that the remainder of the complete reduction $r_{2}:\beta \leadsto \omega$ is equal to the composite complete reduction $\beta \leadsto \sigma \leadsto \omega$, hence the two complete reductions $r_{1}, r_{2}$ of $\alpha_{1}$ are equal.
\end{proof}

\subsection{Proof of Theorem \ref{coherencesmb}}

We are finally ready prove the coherence theorem for symmetric monoidal bicategories.  The strategy here will be to express every 2-cell as a zigzag of complete reductions which Theorem \ref{completered} will then ensure are unique.

\begin{proof}[Proof of Theorem \ref{coherencesmb}]
Let $f,g$ be parallel 1-cells in $\mathcal{S}_{+}$ with the same underlying permutation.  We already know that $f$ and $g$ cannot be isomorphic if they have different underlying permutations, so we will show that there is a unique invertible 2-cell $f \Rightarrow g$.  We know by Proposition \ref{minimal} that there is an isomorphism between $f$ and a minimal 1-cell $f'$ with the same permutation; the same holds for $g$, and since these have the same underlying permutation this gives the existence of some invertible 2-cell $f \Rightarrow g$.  But Theorem \ref{completered} shows that there is a unique complete reduction $f \leadsto f'$, so a unique invertible 2-cell $f \Rightarrow f'$ which only reduces the total number of crossings.  Any 2-cell $\alpha$ in $\mathcal{S}_{+}$ can be factored into a (vertical) composite of 2-cells $\gamma_{k} \beta_{k} \gamma_{k-1} \beta_{k-1} \cdots \gamma_{1} \beta_{1}$ where each $\beta_{i}$ only reduces crossings and each $\gamma_{i}$ only increases crossings.  Thus $\alpha$
can be written as a zigzag of reductions.  Each reduction can be extended to a complete reduction, and all of these can be chosen to have the same target 1-cell $f'$ by Lemma \ref{l=0}.  The resulting triangles of 2-cells are each instances of pairs of complete reductions starting and ending at the same source and target, hence the reductions are equal by Theorem \ref{completered}.  This shows that the corresponding diagrams of 2-cells also commute, so any 2-cell $f \Rightarrow g$ is equal to the unique 2-cell obtained from the unique complete reductions
\[
\xy
{\ar@{~>} (0,0)*+{f}; (20,0)*+{f'} };
{\ar@{~>} (40,0)*+{g.}; (20,0)*+{f'} };
\endxy
\]
\end{proof}

Our proof above can easily be extended to any set of objects seen as a discrete bicategory.

\begin{proof}[Proof of Corollary \ref{coherenceset}]
Changing the set of objects from a terminal set to an arbitrary set merely requires adding labels to the braids; the rest of the argument is easily modified to accommodate this change.
\end{proof}

Finally, we end with a proof of our strictification theorem, Theorem \ref{strictification}.

\begin{proof}[Proof of Theorem \ref{strictification}]
Let $X$ be a symmetric monoidal bicategory.  By \cite{gurski-braid}, the bicategory $\textrm{Gr} X$ can be equipped with the structure of a strict braided monoidal bicategory such that the canonical functors $X \rightarrow \textrm{Gr} X, \textrm{Gr} X \rightarrow X$ become braided monoidal biequivalences.  Recall that an object of $\textrm{Gr}X$ is a string $\mb{x}=(x_{n}, \ldots, x_{1})$ of objects in $X$, possibly empty.  We can equip this braided monoidal structure on $\textrm{Gr} X$ with a symmetric structure by defining $v_{\mb{x}, \mb{y}}$ to be $v_{e(x), e(y)}$ in $X$; the sylleptic and symmetric monoidal bicategory axioms follow immediately from those same axioms in $X$.  Thus we have given $\textrm{Gr}X$ the structure of a symmetric monoidal bicategory whose underlying braided monoidal bicategory is strict.

Now alter the symmetric monoidal structure on $\textrm{Gr}X$ by keeping the same composition and monoidal structure, but by defining $R_{\mb{x}\mb{y}}^{\dot}$ to be the 1-cell represented by $R_{\mb{y}\mb{x}}$ in $X$; similarly, change the unit and counit of the adjunctions to instances of $v^{-1}$ and $v$, respectively.  The same calculations as in Proposition \ref{positivebraidedstructure} show that this gives a new symmetric monoidal structure on $\textrm{Gr}X$, written $\textrm{Gr}^{+}X$, and that the identity functor is a symmetric monoidal biequivalence $1:\textrm{Gr}X \rightarrow \textrm{Gr}^{+}X$.  Composing with the symmetric monoidal biequivalence $X \rightarrow \textrm{Gr}X$ shows that every symmetric monoidal bicategory can be strictified.
\end{proof}

\bibliographystyle{plain}
\bibliography{references}

\end{document}